\documentclass[journal]{IEEEtran}

\usepackage{amssymb}
\usepackage[reqno]{amsmath}
\usepackage{amsthm,color}
\usepackage{amsfonts}
\usepackage{subfigure}
\usepackage{times}
\usepackage[final]{graphicx}
\usepackage{times,amsmath,epsfig}
\usepackage{latexsym,amssymb,mathrsfs}

\usepackage{cite}
\usepackage[ruled]{algorithm2e}
\usepackage{setspace}
\usepackage{color}
\usepackage[noend]{algpseudocode}

\newtheorem{thm}{Theorem}

\newtheorem{cor}{Corollary}

\newtheorem{exmpl}{Example}

\IEEEoverridecommandlockouts

\begin{document}
\title{On the Adversarial Robustness of Robust Estimators}
\author{ Lifeng Lai,~\IEEEmembership{Senior Member,~IEEE} and Erhan Bayraktar \thanks{Lifeng Lai is with Department of Electrical and Computer Engineering, University of California, Davis, CA, 95616. Email: lflai@ucdavis.edu. Erhan Bayraktar is with Department of Mathematics, University of Michigan, Ann Arbor, MI 48104. Email: erhan@umich.edu. The work of L. Lai was supported by the National Science Foundation under grants CCF-17-17943, ECCS-17-11468, CNS-18-24553 and CCF-19-08258. The work of E. Bayraktar was supported in part by the National Science Foundation under grant DMS-1613170 and by the Susan M. Smith Professorship. Copyright (c) 2017 IEEE. Personal use of this material is permitted.  However, permission to use this material for any other purposes must be obtained from the IEEE by sending a request to \textcolor{blue}{pubs-permissions@ieee.org}.}}

\date{\today}
\maketitle
\begin{abstract}
Motivated by recent data analytics applications, we study the adversarial robustness of robust estimators. Instead of assuming that only a fraction of the data points are outliers as considered in the classic robust estimation setup, in this paper, we consider an adversarial setup in which an attacker can observe the whole dataset and can modify all data samples in an adversarial manner so as to maximize the estimation error caused by his attack. We characterize the attacker's optimal attack strategy, and further introduce adversarial influence function (AIF) to quantify an estimator's sensitivity to such adversarial attacks. We provide an approach to characterize AIF for any given robust estimator, and then design optimal estimator that minimizes AIF, which implies it is least sensitive to adversarial attacks and hence is most robust against adversarial attacks. From this characterization, we identify a tradeoff between AIF (i.e., robustness against adversarial attack) and influence function, a quantity used in classic robust estimators to measure robustness against outliers, and design estimators that strike a desirable tradeoff between these two quantities.  
\end{abstract}

\section{Introduction}
Robust estimation is a classic topic that addresses the outlier or model uncertainty issues. In the existing setup, a certain percentage of the data points are assumed to be outliers. Various concepts such as influence function (IF), breakdown point, and change of variance etc were developed to quantify the robustness of estimators against the presence of outliers, please see~\cite{Huber:AMS:64,Huber:Book:09,Hampel:Book:86} and references therein for details. \textcolor{black}{Furthermore, computationally efficient robust algorithms for high dimensional problems were developed in many recent work~\cite{Bhatia:NIPS:15, Prasad:JRSSB:19,Diakonikolas:FOCS:16,Diakonikolas:ICML:2017,Balakrishnan:COLT:17}. }

These concepts are very useful for the classic setup where a \emph{fraction} (up to $50\%$) of data points are outliers while the remaining data come from the true distribution. In this paper, motivated by recent interest in data analytics, we address the issue of adversarial robustness. In a typical data analytics setup, a dataset is stored in a database. If an attacker has access to the database, he can modify \emph{all} data points (i.e., up to $100\%$) in an adversarial manner, and hence the existing results on robust statistics are not directly applicable anymore. This scenario also arises in the adversarial example phenomena in deep neural networks that have attracted significant recent research interests~\cite{Szegedy2013IntriguingPO,Goodfellow:ICLR:15,Carlini:ISSP:17}. In the adversarial example in deep neural networks, by making small but carefully chosen changes on the image, the attacker can mislead neural network to make wrong decisions, even though a human will hardly notice changes on the modified image. Certainly, if the attacker can modify all data and no further restrictions on attacker's capability are imposed, then no meaningful estimator can be constructed (this can be viewed as 100$\%$ of the data are modified in the classic setup). In this paper, we investigate the scenario that the total amount of change measured by $\ell_p$ norm is limited, and we will study how these quantities will affect the estimation performance. Towards this goal, we introduce the concept of adversarial influence function (AIF) to quantify how sensitive an estimator is to adversarial attacks. \textcolor{black}{These types of constraints} are reasonable and are motivated by real life examples. For example, in generating adversarial examples in images~\cite{Szegedy2013IntriguingPO}, the total distortion should be limited, otherwise human eyes will be able to detect such changes. \textcolor{black}{Our problem formulation could also potentially be useful for investigating the robustness of machine learning algorithms under various constraint on the norm of the attack vector, see for example \cite{Mei:AAAI:15}.}  

We first focus on the scenario with a given data set. For this scenario, we characterize the optimal attack vector that the attacker, who observes the whole data set, can employ to maximize the change of estimation result. Using this characterization, we can then analyze AIF of any given estimator. This analysis enables us to design estimators that are robust to adversarial attacks. In particular, from the estimator's perspective, one would like to design an estimator that minimizes AIF, which implies that such an estimator is least sensitive to adversarial attacks and hence is most robust against adversarial attacks. We derive universal lower bounds on AIF and characterize the conditions under which an estimator can achieve this lower bound (and hence is most robustness against adversarial attacks). We then illustrate these results for two specific models: location estimators and scale estimators.

With the results in the given sample scenario, we then extend our study to the population scenario, in which we investigate the behavior of AIF as the number of samples increases. For this case, we identify a tradeoff between robustness against adversarial attacks vs robustness against outliers. In particular, we first characterize the optimal estimator that minimizes AIF. However, the estimator that minimizes AIF has a poor performance in term of IF~\cite{Hampel:Book:86,Hampel:PHD:68}, a quantity that measures robustness against outliers. Realizing this fact, we then formulate optimization problems to design estimators that strike a desirable tradeoff between AIF (i.e., robustness against adversarial attack) and IF (i.e., robustness against outliers). Using tools from calculus of variations~\cite{Burger:Book:03,Kot:Book:14}, we are able to exploit the unique structure of our problems and obtain analytical form of the optimal solution. The obtained solution share similar interpretation as classic robust estimators that it will carefully trim data points that are from the distribution tails. However, the detailed form and thresholds are determined by different criteria.

In the above discussion, we mainly focus on a class of widely used robust estimators: $M$-estimator. However, the developed tools and analysis can be extended to analyze other types of robust estimators. In this paper, we will use $L$-estimator as an example to discuss how to extend the analysis to other types of estimators. 

Our paper is related to a growing list of recent work on adversarial machine learning. Here we give several examples on data poisoning attack that is related to our work. For example,~\cite{Pimentel:ISIT:17} considers an adversarial principal component analysis (PCA) problem. Different from many interesting work on robust PCA~\cite{Candes:JACM:11}, in the model considered in ~\cite{Pimentel:ISIT:17}, an attacker adds an extra data point in an adversarial manner so as to maximize the error of subspace estimated by PCA.~\cite{Jagielski:ISSP:18} investigates data poisoning attack in regression problems, in which the attacker adds data points to the training dataset with the goal of introducing errors into or guiding the results of regression models.~\cite{Biggio:ICML:12} studies an attack that inserts carefully chosen data points to the training set for support vector machine.~\cite{Charikar:STOC:17} considers learning problems from untrusted data. In particular, under the assumption that at least $\alpha$ percent of data points are drawn from a distribution of interest,~\cite{Charikar:STOC:17} considers two frameworks: 1) list-decodable learning, whose goal is to return a list of answers, with the guarantee that at least one of them is accurate; and 2) semi-verified learning, in which one has a small dataset of trusted data that can be leveraged to enable the accurate extraction of information from a much larger but untrusted dataset. \textcolor{black}{There are also a large number of recent work on robust estimators in high dimensions~\cite{Bhatia:NIPS:15, Prasad:JRSSB:19,Diakonikolas:FOCS:16,Diakonikolas:ICML:2017,Balakrishnan:COLT:17}. These papers focus on designing computationally efficient algorithms in the high dimension region.} A major difference between our work and these interesting work is that the existing work assume that a certain percentage of data points are not compromised, while in our work all data points could be compromised. \textcolor{black}{Our paper is also related to a recent interesting paper \cite{Suggala:COLT:19} that studies the problem of robust linear regression with response variable corruptions. \cite{Suggala:COLT:19} considers an oblivious adversary model, in which the adversary changes a fraction of responses without knowledge of the data, and provides a nearly linear time estimator that is consistent even when the majority of the data is corrupted. In our paper, the modification introduced by the attack can be dependent on the whole dataset. } 

The remainder of the paper is organized as follows. In Section~\ref{sec:intro}, we introduce our problem formulation. In Section~\ref{sec:background}, we introduce the necessary background. In Section~\ref{sec:M}, we investigate AIF for the given sample scenario. In Section~\ref{sec:pop}, we consider the population scenario. We extend the study to $L$-estimator in Section~\ref{sec:L}. Numerical examples are given in Section~\ref{sec:example}. Finally, we offer concluding remarks in Section~\ref{sec:con}.

\section{Model}~\label{sec:intro}



We consider an adversarially robust parameter estimation problem in which the adversary has access to the whole dataset. In particular, we have a given data set $\mathbf{x}=\{x_1,\cdots,x_N\}$, in which $x_n$ are i.i.d realizations of random variable $X\in \mathbb{R}$ that has cumulative density function (cdf) $F_{\theta}(x)$ with unknown parameter $\theta\in \mathbb{R}$. We will use $f_{\theta}(x)$ to denote the corresponding probability density function (pdf). From this given data set, we would like to estimate the unknown parameter $\theta$. However, as the adversary has access to the whole dataset, it will modify the data to $\mathbf{x}^{\Delta}=\mathbf{x}+\Delta \mathbf{x}:=\{x_1+\Delta x_1, \cdots,x_N+\Delta x_N\}$, in which $\Delta \mathbf{x}=\{\Delta x_1, \cdots,\Delta x_N\}$ is the attack vector chosen by the adversary after observing $\mathbf{x}$. We will discuss the attacker's optimal attack strategy in choosing $\Delta \mathbf{x}$ in the sequel. In the classic robust estimation setup, it is typically assume that some percentage (up to $50\%$) of the data points are outliers, that is some entries in $\Delta \mathbf{x}$ are nonzero while the remainders are zero. In this work, we consider the case where the attacker can modify all data points, which is a more suitable setup for recent data analytical applications. However, certain restrictions need to be put on $\Delta \mathbf{x}$, otherwise the estimation problem will not be meaningful. In this paper, we assume that 
\begin{eqnarray}
\frac{1}{N}||\Delta \mathbf{x}||_p^p\leq \eta^p,\label{eq:norconstranit}
\end{eqnarray}
in which $||\cdot||_p$ is the $\ell_p$ norm. The normalization factor $N$ implies that the per-dimension change (on average) is upper-bound by $\eta^p$. As mentioned in the introduction, this type of constraints are reasonable and are motivated by real life examples. 
The classic setup can be viewed as a special case of our formulation by letting $p\rightarrow 0$, i.e., the classic setup has constraint on the total number of data points that the attacker can modify.

Following notation used in robust statistics \cite{Huber:Book:09,Hampel:Book:86}, we will use $T_N(\mathbf{x})$ to denote an estimator. For a given estimator $T_N$, we would like to characterize how sensitive the estimator is with respect to the adversarial attack. In this paper, we consider a scenario where the goal of the attacker is to maximize the \textcolor{black}{deviation in the estimator's output caused by the attack}. In particular, the attacker aims to choose $\Delta\mathbf{x}$ by solving the following optimization problem
\begin{eqnarray}\label{eq:optimalintro}
\max\limits_{\Delta\mathbf{x}} && |T_N(\mathbf{x}+\Delta\mathbf{x})-T_N(\mathbf{x})|,\\
\text{s.t.}&& \frac{1}{N}||\Delta \mathbf{x}||_p^p\leq \eta^p.\nonumber
\end{eqnarray}

We use $\Delta T_N(\mathbf{x})$ to denote the optimal value obtained from the optimization problem~\eqref{eq:optimalintro}, and define the adversarial influence function (AIF) of estimator $T_N$ at $\mathbf{x}$ under $\ell_p$ norm constraint as
\begin{eqnarray}
\text{AIF}(T_N,\mathbf{x},p)=\lim_{\eta\downarrow 0}\frac{\Delta T_N(\mathbf{x})}{\eta}.\nonumber
\end{eqnarray}
This quantity, a generalization of the concept of IF used in classic robust estimation (we will briefly review IF in Section~\ref{sec:background}), quantifies the rate at which the attacker can introduce estimation error through its attack. 

From the defender's perspective, the smaller AIF is, the more robust the estimator is. In this paper, building on the characterization of $\text{AIF}(T_N,\mathbf{x},p)$, we will characterize the optimal estimator $T_N$, among a certain class of estimators $\mathcal{T}$, that minimizes $\text{AIF}(T_N,\mathbf{x},p)$. In particular, we will investigate 
\begin{eqnarray}
\min\limits_{T_N\in \mathcal{T}} \text{AIF}(T_N,\mathbf{x},p).\nonumber
\end{eqnarray}
We will show that, for certain class of $\mathcal{T}$, the optimal $T_N$ is independent of $\mathbf{x}$ and $p$, which is a very desirable property.

Note that $\text{AIF}(T_N,\mathbf{x},p)$ depends on the data realization $\mathbf{x}$. Based on the characterization of AIF for a given data realization $\mathbf{x}$ of length $N$, we will then study the population version of AIF where each entry of $\mathbf{X}=\{X_1,\cdots,X_N\}$ is i.i.d generated by $F_{\theta}$. We will examine the behavior of $\text{AIF}(T_N,\mathbf{X},p)$ as $N$ increases. Following the convention in robust statistics, we will assume that there exists a functional $T$ such that 
\begin{eqnarray}
T_N(\mathbf{X})\rightarrow T(F_{\theta})\label{eq:functional}
\end{eqnarray}
in probability as $N\rightarrow \infty$.
We will see that for a large class of estimators $\text{AIF}(T_N,\mathbf{X},p)$ has a well-defined limit as $N\rightarrow \infty$. We will use $\text{AIF}(T,F_{\theta},p)$ to denote this limit when it exists.


Similarly, from the defense's perspective, we would like to design an estimator that is least sensitive to the adversarial attack. Again, we will characterize the optimal estimator $T$, among a certain class of estimators $\mathcal{T}$, that minimizes $\text{AIF}(T,F_{\theta},p)$. That is, for a certain class of estimators $\mathcal{T}$, we will solve
\begin{eqnarray}
\min\limits_{T\in \mathcal{T}}  \text{AIF}(T,F_{\theta},p).\label{eq:AIF}
\end{eqnarray}

It will be clear in the sequel that the solution to the optimization problem~\eqref{eq:AIF}, even though is robust against adversarial attacks, has poor performance in guarding against outliers. This motivates us to design estimators that strike a desirable tradeoff between these two robustness measures. In particular, we will solve~\eqref{eq:AIF} with an additional constraint on IF. We will need to use tools from calculus of variations for this purpose.

\textcolor{black}{We note that in this paper, we focus on the scalar case (i.e., $X$ and $\theta$ are scalers). The problem formulation and analysis can be extended (with additional technical developments) to the more general vector case (including joint location-scale estimation and robust regression etc.). The corresponding results are reported in~\cite{Bayraktar:SMDS:19}.}
\section{Background}\label{sec:background}
In this section, we briefly review results from classic robust estimator literature that are closely related to our study.

\subsection{Influence Function (IF)}
As mentioned above, in the classic robust estimation setup, it is assumed that a fraction $\eta$ of data points are outliers, while the remainder of data points are generated from the true distribution $F_{\theta}$. For a given estimator $T$, the concept of IF introduced by Hampel~\cite{Hampel:PHD:68} is defined
\begin{eqnarray}
\text{IF}(x,T,F_{\theta})=\lim_{\eta\downarrow 0}\frac{T((1-\eta)F_{\theta}+\eta \delta_{x})-T(F_{\theta})}{\eta}.\nonumber
\end{eqnarray}
In this definition, $\delta_{x}$ is a distribution that puts mass 1 at point $x$, $T(F_{\theta})$, introduced in~\eqref{eq:functional}, is the obtained estimate when all data points are generated i.i.d from $F_{\theta}$, and $T((1-\eta)F_{\theta}+\eta \delta_{x})$ is the obtained estimate when $1-\eta$ fraction of data points are generated i.i.d from $F_{\theta}$ while $\eta$ fraction of the data points are at $x$. Hence, $\text{IF}(x,T,F_{\theta})$ measures the influence of having outliers at point $x$ as $\eta\downarrow 0$.  

To measure the influence of the worst outliers,~\cite{Hampel:PHD:68} then further introduced the concept of gross-error sensitivity of $T$ by taking $\sup$ over the absolute value of $\text{IF}(x,T,F_{\theta})$:
\begin{eqnarray}
\gamma^*(T,F_{\theta})=\sup_{x}|\text{IF}(x,T,F_{\theta})|.\nonumber
\end{eqnarray}

Intuitively speaking, $\gamma^*(T,F_{\theta})$ can be viewed as the solution of our problem setup for the special case of $p=0$.

The values of $\text{IF}(x,T,F_{\theta})$ and $\gamma^*(T,F_{\theta})$ have been characterized for various class of estimators. Furthermore, under certain conditions, optimal estimator $T$ that minimizes these quantities have been established. Some of these results will be introduced in later sections. More details can be found in~\cite{Huber:Book:09,Hampel:Book:86}.

\subsection{$M$-Estimator}

In this paper, we will mainly focus on a class of commonly used estimator in robust statistic: $M$-estimator~\cite{Huber:AMS:64}, in which one obtains an estimate $T_N(\mathbf{x})$ of $\theta$ by solving 
\begin{eqnarray}\label{eq:Mestimate}
\sum\limits_{n=1}^N \psi(x_n,T_N)=0.
\end{eqnarray}

Here $\psi(x_n,\theta)$ is a function of data $x_n$ and parameter $\theta$ to be estimated. Different choices of $\psi$ lead to different robust estimators. For example, the most likely estimator (MLE) can be obtained by setting $\psi=-f_{\theta}^{'}/f_{\theta}$. \textcolor{black}{$M$-estimator can also be defined as the solution of an optimization problem. The formulation in~\eqref{eq:Mestimate} and the optimization formulation have certain relationship, but they are not always equivalent. Please refer to Chapter 2.3a of~\cite{Hampel:Book:86} for detailed discussion. }

As the form of $\psi$ determines $T_N$, in the remainder of the paper, we will use $\psi$ and $T_N$ interchangeably. For example, we will denote $\text{AIF}(T_N,\mathbf{x},p)$ as $\text{AIF}(\psi,\mathbf{x},p)$. Similarly, we will denote $\text{IF}(x,T,F_{\theta})$ as $\text{IF}(x,\psi,F_{\theta})$.

It is typically assumed that $\psi(x,\theta)$ is continuous and almost everywhere differentiable. This assumption is valid for all $\psi$'s that are commonly used. It is also typically assume the estimator is Fisher consistent~\cite{Hampel:Book:86}:
\begin{eqnarray}
\mathbb{E}_{F_\theta}[\psi(X,\theta)]=0,\label{eq:Fisher}
\end{eqnarray}
in which $\mathbb{E}_{F_\theta}$ means expectation under $F_{\theta}$. Intuitively speaking, this implies that the true parameter $\theta$ is the solution of the $M$-estimator if there are increasingly more i.i.d. data points generated from $F_{\theta}$.

For $M$-estimator, $\text{IF}(x,\psi, F_{\theta})$ was shown to be  
\begin{eqnarray}
\text{IF}(x,\psi, F_{\theta})=\frac{\psi(x,T(F_{\theta}))}{-\int \frac{\partial}{\partial \theta}[\psi(y,\theta)]_{\theta=T(F_{\theta})}\text{d}F_{\theta}(y)},\nonumber
\end{eqnarray}
see (2.3.5) of~\cite{Hampel:Book:86}.

\section{The Fixed Sample Case}\label{sec:M}


In this section, we focus on analyzing $\text{AIF}(\psi, \mathbf{x},p)$ for a given dataset $\mathbf{x}$. We will extend the study to the population case and analyze $\text{AIF}(\psi, F_{\theta},p)$ in Section~\ref{sec:pop}.

\subsection{General $\psi$}

We will first characterize $\text{AIF}(\psi, \mathbf{x},p)$ for general $\psi$, and will then specialize the results to specific problems in later sections. For any given $\psi$ that is continuous and almost everywhere differentiable, we have the following theorem that characterizes $\text{AIF}(\psi,\mathbf{x},p)$.
\begin{thm}\label{thm:asc}
	
	When $p=1$,
	\begin{eqnarray}
	\text{AIF}(\psi,\mathbf{x},1)=\frac{\left|\frac{\partial }{\partial x}[\psi]_{x=x_{n^*}, \theta=T_N}\right|}{\left|\frac{1}{N}\sum\limits_{n=1}^{N}\frac{\partial }{\partial \theta}[\psi]_{x=x_n,\theta=T_N}\right|},\nonumber
	\end{eqnarray}
	where \begin{eqnarray} 
	n^*=\arg\max\limits_{n} \left |\frac{\partial }{\partial x}[\psi]_{x=x_n, \theta=T_N}\right|.\label{eq:nstar}
	\end{eqnarray}
	
	For $p>1$, we have
	\begin{eqnarray}
	\text{AIF}(\psi,\mathbf{x},p)=\frac{\left(\frac{1}{N}\sum\limits_{n=1}^N\left|\frac{\partial }{\partial x}[\psi]_{x=x_n,\theta=T_N}\right|^{\frac{p}{p-1}}\right)^{\frac{p-1}{p}}}{\left|\frac{1}{N}\sum\limits_{n=1}^N \frac{\partial}{\partial \theta}[\psi]_{x=x_n,\theta=T_N}\right|}.\nonumber
	\end{eqnarray}
	
\end{thm}
\begin{proof}
	Please see Appendix~\ref{app:asc} for detailed proof.
\end{proof}

\textcolor{black}{In this theorem, we characterize the result for $p\geq 1$. Ideally, one would like to consider the case with $p<1$, but this will result in a non-convex optimization, which precludes us from obtaining a closed form solution. }


From Theorem~\ref{thm:asc}, we can characterize the form of $\psi$ that leads to the smallest $\text{AIF}$, i.e., the most robust $M$-estimator against adversarial attacks.
\begin{cor}
	\begin{eqnarray}
	\text{AIF}(\psi,\mathbf{x},p)\geq\frac{\frac{1}{N}\sum\limits_{n=1}^N\left|\frac{\partial }{\partial x}[\psi]_{x=x_n,\theta=T_N}\right|}{\left|\frac{1}{N}\sum\limits_{n=1}^N \frac{\partial}{\partial \theta}[\psi]_{x=x_n,\theta=T_N}\right|},\nonumber
	\end{eqnarray}	
	and the equality holds when 
	\begin{eqnarray}
	\left|\frac{\partial }{\partial x}[\psi]_{x=x_1,\theta=T_N}\right|=\cdots=\left|\frac{\partial }{\partial x}[\psi]_{x=x_N,\theta=T_N}\right|.\nonumber
	\end{eqnarray}
\end{cor}
\begin{proof}
	For $p>1$, it is easy to check that $x^{(p-1)/p}$ is a concave function when $x\geq 0$. Hence, using Jensen's inequality, we have
	\begin{eqnarray}
	&&\left(\frac{1}{N}\sum\limits_{n=1}^N\left|\frac{\partial }{\partial x}[\psi]_{x=x_n,\theta=T_N}\right|^{\frac{p}{p-1}}\right)^{\frac{p-1}{p}}\nonumber\\
	&&\hspace{5mm}\geq \frac{1}{N}\sum\limits_{n=1}^N\left|\frac{\partial }{\partial x}[\psi]_{x=x_n,\theta=T_N}\right|,\nonumber
	\end{eqnarray}
	and the equality holds when $\left|\frac{\partial }{\partial x}[\psi]_{x=x_n,\theta=T_N}\right|$ is a constant with respect to $n$.
\end{proof}

This corollary implies that, from defender's perspective, we should design $\psi(x,\theta)$ such that $\left|\frac{\partial}{\partial x}[\psi]\right|$ is constant in $x$. It is also interesting that, this result holds for any value of $p$. And hence we can design an estimator without knowledge about which constraint the attacker is using. 

\subsection{Specific Estimators}
To illustrate the results obtained above, we specialize results to location estimators and scale estimators.
\subsubsection{Location Estimator}

For location estimator models, $F_{\theta}(x)=F_{0}(x-\theta)$, and hence it is natural to use $\psi(x,\theta)=\psi(x-\theta)$, see~\cite{Huber:Book:09,Hampel:Book:86}. For this model, it is easy to check that
\begin{eqnarray}
\frac{\partial }{\partial x}[\psi]_{x=x_n, \theta=T_N}=\psi^{'}(x_n-T_N),\nonumber\\
\frac{\partial }{\partial \theta}[\psi]_{x=x_n, \theta=T_N}=-\psi^{'}(x_n-T_N).\nonumber
\end{eqnarray}
Plugging these two equations in the AIF expressions in Theorem~\ref{thm:asc}, for the case with $p=1$, we have 
\begin{eqnarray}
\text{AIF}(\psi,\mathbf{x},1)=\frac{\left|N\psi^{'}(x_{n^*}-T_N)\right|}{\left|\sum\limits_{n=1}^{N}\psi^{'}(x_n-T_N)\right|}\geq 1,\nonumber
\end{eqnarray}
for which the equality holds when $\psi^{'}(x_n-T_N)$ is a constant with respect to $n$. 

For the case with $p>1$, we have
\begin{eqnarray}
\text{AIF}(\psi,\mathbf{x},p)&=&\frac{\left(\frac{1}{N}\sum\limits_{n=1}^N\left|\psi^{'}(x_n-T_N)\right|^{\frac{p}{p-1}}\right)^{\frac{p-1}{p}}}{\left|\frac{1}{N}\sum\limits_{n=1}^N\psi^{'}(x_n-T_N)\right|}\label{eq:AIFp}\\
&\overset{(a)}{\geq}&\frac{\frac{1}{N}\sum\limits_{n=1}^N\left|\psi^{'}(x_n-T_N)\right|}{\left|\frac{1}{N}\sum\limits_{n=1}^N\psi^{'}(x_n-T_N)\right|}\geq 1,\nonumber
\end{eqnarray}
in which (a) is due to Jensen's inequality. Both inequalities will hold if $\psi^{'}(x_n-T_N)$ is a constant in $n$. 

\begin{exmpl}
	Consider an estimator with $\psi(x_n-T_N)=x_n-T_N$. This estimator is simply the empirical sample mean. It is easy to see that $\psi^{'}(x)$ is a constant in $n$, which implies that this choice of $\psi$ has $\text{AIF}(\psi,\mathbf{x},p)=1$. It achieves the lower bound established above, regardless of the value of $\mathbf{x}$ and $p$. Hence, it is the most robust estimator against adversarial attacks. However, as we will discuss in Section~\ref{sec:pop}, this choice of $\psi$ is not robust against outliers. In Section~\ref{sec:pop}, we will design estimators that strike a desirable balance between robustness against outliers and robustness against adversarial attacks.
\end{exmpl} 

\begin{exmpl}
	Consider the Huber estimator~\cite{Huber:AMS:64} with $$\psi(x_n-T_N)=\min\{b,\max \{x_n-T_N,-b\}\},$$ parameterized by a parameter $0<b<\infty$. Using~\eqref{eq:AIFp}, it is easy to check that $\text{AIF}(\psi,\mathbf{x},p)=\sqrt{1/\beta}$, in which $\beta$ is the \textcolor{black}{proportion of points} in $\mathbf{x}$ such that $|x_n-T_N|< b$. It is clear that Huber estimator, while being more robust against outliers~\cite{Huber:Book:09}, is less robust against adversarial attacks than the empirical mean estimator.
\end{exmpl}

\subsubsection{Scale Estimator}
The scale model~\cite{Huber:Book:09,Hampel:Book:86} is given by $F_{\theta}(x)=F_1(x/\theta)$, and it is typical to consider
$\psi(x,\theta)=\psi(x/\theta).$ It is easy to check that for $\psi$ with this form, we have
\begin{eqnarray}
\frac{\partial }{\partial x}[\psi]_{x=x_n, \theta=T_N}=\frac{\psi^{'}(x_n/T_N)}{T_N},\nonumber\\
\frac{\partial }{\partial \theta}[\psi]_{x=x_n, \theta=T_N}=\frac{-x_n\psi^{'}(x_n/T_N)}{T_N^2}.\nonumber
\end{eqnarray}

Using Theorem~\ref{thm:asc}, for the case with $p=1$, we obtain
\begin{eqnarray}
\hspace{4mm}\text{AIF}(\psi,\mathbf{x},1)&=&\frac{\left|N\frac{\psi^{'}(x_{n^*}/T_N)}{T_N}\right|}{\left|\sum\limits_{n=1}^{N}\frac{-x_n\psi^{'}(x_n/T_N)}{T_N^2}\right|}\nonumber\\&=& \frac{\left|\psi^{'}(x_{n^*}/T_N)\right|}{\left|\frac{1}{N}\sum\limits_{n=1}^{N}x_n/T_N\psi^{'}(x_n/T_N)\right|}\label{eq:scale1},
\end{eqnarray}
in which $n^*$ is defined in~\eqref{eq:nstar}.

When $p>1$, we have
\begin{eqnarray}
\text{AIF}(\psi,\mathbf{x},p)&=&\frac{\left(\frac{1}{N}\sum\limits_{n=1}^N\left|\frac{\psi^{'}(x_n/T_N)}{T_N}\right|^{\frac{p}{p-1}}\right)^{\frac{p-1}{p}}}{\left|\frac{1}{N}\sum\limits_{n=1}^N \frac{-x_n\psi^{'}(x_n/T_N)}{T_N^2}\right|}\nonumber \\&=&\frac{\left(\frac{1}{N}\sum\limits_{n=1}^N\left|\psi^{'}(x_n/T_N)\right|^{\frac{p}{p-1}}\right)^{\frac{p-1}{p}}}{\left|\frac{1}{N}\sum\limits_{n=1}^N x_n/T_N\psi^{'}(x_n/T_N)\right|}.\label{eq:scalep}
\end{eqnarray}

\begin{exmpl}
	Consider MLE for variance of zero-mean Gaussian random variables, which corresponds to $\psi(x)=-x(\phi^{'}(x)/\phi(x))-1=x^2-1$. Here, $\phi(x)$ is the pdf of zero mean variance one Gaussian random variable. For this choice of $\psi$, we have $T_N=\frac{1}{N}\sum_{n=1}^{N}x_n^2$ and $\psi^{'}(x_n/T_N)=2x_n/T_N$. Plugging these values into~\eqref{eq:scale1}, we obtain
	\begin{eqnarray}
	\text{AIF}(\psi,\mathbf{x},1)=|x_n^*|.\nonumber
	\end{eqnarray}
	Using~\eqref{eq:scalep}, we have
	\begin{eqnarray}
	\text{AIF}(\psi,\mathbf{x},p)=\left(\frac{1}{N}\sum_{n=1}^N|x_n|^{\frac{p}{p-1}}\right)^{\frac{p-1}{p}},\nonumber
	\end{eqnarray}
	from which we know that when $p=2$, $\text{AIF}(\psi,\mathbf{x},2)=\sqrt{T_N}=\sqrt{\frac{1}{N}\sum_{n=1}^{N}x_n^2}.$
	
\end{exmpl} 

\section{Population Case}\label{sec:pop}
With the results on the fixed dataset case, we now consider the population version where $X_n$ are i.i.d from $F_{\theta}$, and analyze the behavior of AIF as $N\rightarrow\infty$. Following the convention in classic \textcolor{black}{robust statistics literature}, we will focus on the case in which the estimator is Fisher consistent as defined in~\eqref{eq:Fisher}.

\color{black}
It has been shown in Theorem 2.4 of~\cite{Huber:Book:09} that, under certain mild regularity conditions, $T_N\overset{a.s.}{\rightarrow}\theta$. In the following, we will need the following additional regularity conditions:
\begin{itemize}
	\item $\psi^{'}(x)$ is continuous functions. 
	\item There exist a function $K(x)$ such that $|\psi^{'}(x)|\leq K(x)$, $|x\psi_1^{'}(x)|\leq K(x)$, and $\mathbb{E}_{F_{\theta}}[K(X)]<\infty$.
\end{itemize} 
The conditions here are slightly stronger than those conditions needed for the strong law of large numbers, as we will need to use the uniform strong law of large numbers (see Theorem 16 (a)~\cite{Ferguson:Book:96}). Under these regularity assumptions, using the uniform strong law of large numbers, Slutsky Theorem (see Chapter 6 of~\cite{Ferguson:Book:96}) and the fact that $T_N\overset{a.s.}{\rightarrow}\theta$, as $N\rightarrow \infty$ we have 
\begin{eqnarray}
\frac{1}{N}\sum\limits_{n=1}^{N}\frac{\partial }{\partial \theta}[\psi]_{x=x_n,\theta=T_N}\overset{a.s.}{\rightarrow}\mathbb{E}_{F_{\theta}}\left[\frac{\partial }{\partial \theta}[\psi](X,\theta)\right].
\end{eqnarray}
Furthermore, using Proposition 3 of~\cite{Croux:SPL:98}, as $N\rightarrow \infty$ in Theorem~\ref{thm:asc}, 
we have
\begin{eqnarray}
\left|\frac{\partial }{\partial x}[\psi]_{x=x_{n^*}, \theta=T_N}\right|\overset{a.s.}{\rightarrow}\max\limits_{x}\left|\frac{\partial }{\partial x}[\psi](x, \theta)\right|.
\end{eqnarray}
\color{black}
As the result,
\begin{eqnarray}
\text{AIF}(\psi,\mathbf{x},1)&=&\frac{\left|\frac{\partial }{\partial x}[\psi]_{x=x_{n^*}, \theta=T_N}\right|}{\left|\frac{1}{N}\sum\limits_{n=1}^{N}\frac{\partial }{\partial \theta}[\psi]_{x=x_n,\theta=T_N}\right|}\nonumber\\
&\overset{{a.s.}}{\rightarrow}&\frac{\max\limits_{x}\left|\frac{\partial }{\partial x}[\psi](x, \theta)\right|}{\left|\mathbb{E}_{F_{\theta}}\left[\frac{\partial }{\partial \theta}[\psi](X,\theta)
	\right]\right|}\nonumber\\&:=&\text{AIF}(\psi,F_{\theta},1).\label{eq:pop1}
\end{eqnarray}

For $p>1$, we have
\begin{eqnarray}
\text{AIF}(\psi,\mathbf{x},p)&=&\frac{\left(\frac{1}{N}\sum\limits_{n=1}^N\left|\frac{\partial }{\partial x_n}[\psi]_{x=x_n,\theta=T_N}\right|^{\frac{p}{p-1}}\right)^{\frac{p-1}{p}}}{\left|\frac{1}{N}\sum\limits_{n=1}^N \frac{\partial}{\partial \theta}[\psi]_{x=x_n,\theta=T_N}\right|}\nonumber\\
&\overset{{a.s.}}{\rightarrow}&\frac{\left(\mathbb{E}_{F_{\theta}}\left[\left|\frac{\partial }{\partial x}[\psi](X,\theta)\right|^{\frac{p}{p-1}}\right]\right)^{\frac{p-1}{p}}}{\left|\mathbb{E}_{F_{\theta}}\left[ \frac{\partial}{\partial \theta}[\psi](X,\theta)\right]\right|}\nonumber\\&:=&\text{AIF}(\psi,F_{\theta},p).\label{eq:popp}
\end{eqnarray}

\subsection{Location Estimator}
We now specialize the results to the location model mentioned above. We will first characterize $\psi$ that minimizes $\text{AIF}(\psi, F_{\theta}, p)$. We will then discuss the tradeoff between the robustness to outliers and robustness to adversarial attacks, and will characterize the optimal $\psi$ that achieves this tradeoff. In the location estimator, we will assume $\psi(x,\theta)$ is monotonic in $\theta$, which will satisfy the regularity conditions established in~\cite{Huber:Book:09}.   

\subsubsection{Minimizing $\text{AIF}(\psi, F_{\theta}, p)$}\label{sec:minAIF}
For $p=1$, using~\eqref{eq:pop1}, we have 
\begin{eqnarray}
\text{AIF}(\psi,F_{\theta},1)=\frac{\max\limits_{x}\left|\psi^{'}(x-\theta)\right|}{\left|\mathbb{E}_{F_{\theta}}[\psi^{'}(X-\theta)]\right|}.\nonumber
\end{eqnarray}

For $p>1$, using~\eqref{eq:popp}, we obtain
\begin{eqnarray}
\text{AIF}(\psi, F_{\theta}, p)=\frac{\left(\mathbb{E}_{F_{\theta}}\left[\left|\psi^{'}(X-\theta)\right|^{\frac{p}{p-1}}\right]\right)^{\frac{p-1}{p}}}{\left|\mathbb{E}_{F_{\theta}}\left[ \psi^{'}(X-\theta)\right]\right|}.\label{eq:nlarge}
\end{eqnarray}

In particular, for $p=2$, we have
\begin{eqnarray}
\text{AIF}(\psi,F_{\theta},2)=\sqrt{\frac{\mathbb{E}_{F_{\theta}}[\psi^{'}(X-\theta)^2]}{(\mathbb{E}_{F_{\theta}}[\psi^{'}(X-\theta)])^2}}.\nonumber
\end{eqnarray}

From~\eqref{eq:nlarge} and using Jensen's equality, we have
\begin{eqnarray}
\text{AIF}(\psi, F_{\theta}, p)\geq 1,\nonumber
\end{eqnarray}
for which the equality holds when $\psi^{'}(x-\theta)$ is constant in $x$. 

\subsubsection{Tradeoff between $AIF(\psi, F_{\theta},p)$ and $\gamma^*(\psi, F_{\theta})$}

From (2.3.12) of~\cite{Hampel:Book:86}, we know that the influence function of the location estimator specified by $\psi$ is
\begin{eqnarray}
\text{IF}(x,\psi, F_{\theta})=\frac{\psi(x-\theta)}{\mathbb{E}_{F_{\theta}}[\psi^{'}(X-\theta)]},\nonumber
\end{eqnarray}
and hence
\begin{eqnarray}
\gamma^*(\psi, F_{\theta})=\sup_{x}\left|\frac{\psi(x-\theta)}{\mathbb{E}_{F_{\theta}}[\psi^{'}(X-\theta)]}\right|.\nonumber
\end{eqnarray}
As a result, if $\psi^{'}(X-\theta)$ is a constant that minimizes $\text{AIF}(\psi, F_{\theta}, p)$ as discussed in Section~\ref{sec:minAIF}, then $\gamma^*(\psi, F_{\theta})$ might go to $\infty$, especially for those distributions with unbounded support. To achieve a desirable tradeoff between robustness to outliers (i.e., $\gamma^*(\psi, F_{\theta})$ is small) and robustness to adversarial attacks (i.e., $\text{AIF}(\psi, F_{\theta}, p)$ is small), in the following, we characterize the optimal estimator that minimizes  $\text{AIF}(\psi, F_{\theta}, p)$ subject to a constraint on $\gamma^*(\psi, F_{\theta})$.
\begin{eqnarray}\label{eq:opt}
\min&& \text{AIF}(\psi,F,2)
\\
\text{s.t.}&& 
\gamma^*(\psi, F_{\theta})\leq \xi,\label{eq:gammabound}\\
&&\mathbb{E}_{F_{\theta}}[\psi(X-\theta)]=0,\label{eq:fisherco}\\
&& \psi^{'}(x)\geq 0,\nonumber
\end{eqnarray}
in which constraint~\eqref{eq:gammabound} implies that $\gamma^*(\psi, F_{\theta})$ is upperbounded by a positive constant $\xi$, constraint~\eqref{eq:fisherco} implies that $\psi$ is Fisher consistent, and the last constraint comes from the condition that $\psi$ is monotonic in $\theta$. 


\textcolor{black}{For location estimator, $f_{\theta}(x)=f_0(x-\theta)$, so all quantities in~\eqref{eq:opt} remain the same by assuming $\theta=0$~\cite{Hampel:Book:86}. Hence, in the following, we will solve this optimization problem assuming $\theta=0$. Once the optimal form of $\psi$ for $\theta=0$ is characterized, we can obtain the estimate of $\theta$ by solving $\sum_{n=1}^{N}\psi(x-T_N)=0$ for the general case when $\theta\neq 0$. }
\begin{thm}\label{thm:tradeoff}
	The solution to the optimization problem~\eqref{eq:opt} has the following structure:
	\begin{itemize}
		\item $\psi^{'}(x)$ satisfies
		\begin{eqnarray}
		&&\hspace{-13mm}\psi^{'}(x)= \\\nonumber
		&&\hspace{-13mm}\left\{\begin{array}{cc}
		\nu^*-\frac{ \vartheta_2^*+(\vartheta_1^*-\vartheta_2^*)F_0(x)}{ f_0(x)}, &\hspace{-2mm} \nu^* f_0(x)> \vartheta_2^*+(\vartheta_1^*-\vartheta_2^*)F_0(x);\\
		0,& \text{otherwise,}
		\end{array}\right.\label{eq:locationopt}
		\end{eqnarray}
		in which the parameters $\nu^*$, $\vartheta_1^*\geq 0$ and $\vartheta_2^*\geq 0$ are parameters chosen to satisfy the following conditions
		\begin{eqnarray}
		\mathbb{E}_{F_{0}}[\psi^{'}(X)]=1,\label{eq:1}\\
		\vartheta_1^{*}\left(\int \psi^{'}(x)F_0(x)\text{d}x-\xi\right)=0,\label{eq:2}\\
		\vartheta_2^{*}\left(\mathbb{E}_{F_0}\left[\int_{-\infty}^X\psi^{'}(t)\text{d}t\right]-\xi\right)=0,\label{eq:3}
		\end{eqnarray}
		along with $\int \psi^{'}(x)F_0(x)\text{d}x\leq \xi$ and $\mathbb{E}_{F_0}\left[\int_{-\infty}^X\psi^{'}(t)\text{d}t\right]\leq \xi$.
		\item $\psi(-\infty)$ is set as $-\mathbb{E}_{F_0}\left[\int_{-\infty}^X\psi^{'}(t)\text{d}t\right]$.
		
	\end{itemize}
	
\end{thm}

\begin{proof}
	Please see Appendix~\ref{app:tradeoff} for details.

\end{proof}

The condition $\nu^*f_0(x)> \vartheta_1^* F_0(x)+\vartheta_2^*(1-F_0(x))$ has a natural interpretation. It will trim data points from the tails. In particular, when $x$ is left tail (i.e. $F_0$ is small), $1-F_0$ will be close to 1. On the other hand, when $x$ is in the right tail (i.e. $1-F_0$ is small), $F_0$ will be close to 1. In these regions, $\psi^{'}=0$ if the corresponding $f_0(x)$ is small. Figure~\ref{fig:guassian} illustrates the scenario for estimating the mean of Gaussian variables for the case assuming $\vartheta_1^*>\vartheta_2^*$. It is easy to check that, in this example, if $\nu^*>2\pi (\vartheta_1^*+\vartheta_2^*)$, there exist $a$ and $b$ such that $\psi^{'}(x)=0$ when $x<a$ or $x>b$. Correspondingly, $\psi(x)$ is given as

\begin{eqnarray}
\psi(x)=\left\{\begin{array}{cc}
\xi & x\geq b\\
-\xi+\int_{a}^{x}\psi^{'}(t)\text{d}t & a<x<b\\
-\xi & x<a
\end{array}\right..\nonumber
\end{eqnarray}

\begin{figure}[!htb]
	\centering
	\includegraphics[width=0.55\textwidth]{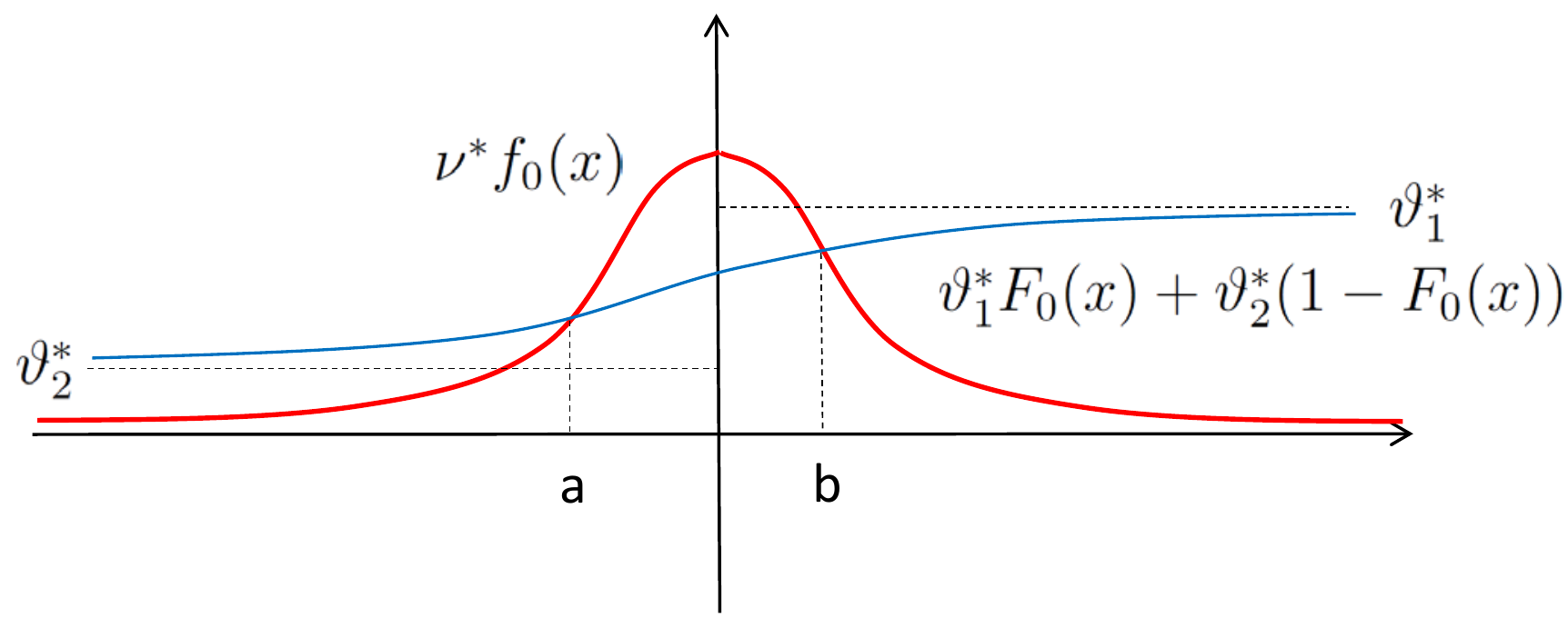}
	\caption{Gaussian mean example}
	\label{fig:guassian}	
\end{figure}

\subsection{Scale Estimator}
We now specialize the results to the scale model where $F_{\theta}(x)=F_{1}(x/\theta)$. For this model, it is natural to consider $\psi(x,\theta)=\psi(x/\theta)$~\cite{Huber:Book:09,Hampel:Book:86}. Similar to the location model, we will first characterize $\psi$ that minimizes $\text{AIF}(\psi, F_{\theta}, p)$. We will then discuss the tradeoff between the robustness to outliers and robustness to adversarial attacks, and will characterize the optimal $\psi$ that achieves this tradeoff. 

For the case with $p=1$, using~\eqref{eq:pop1}, we obtain
\begin{eqnarray}
\text{AIF}(\psi,\mathbf{x},1)&=& \frac{\left|NT_N\psi^{'}(x_{n^*}/T_N)\right|}{\left|\sum\limits_{n=1}^{N}x_n\psi^{'}(x_n/T_N)\right|}\overset{\text{a.s.}}{\rightarrow}\frac{\max\limits_{x}\left|\theta\psi^{'}(x/\theta)\right|}{\left|\mathbb{E}_{\theta}[X\psi^{'}(X/\theta)]\right|}\nonumber\\&:=&\text{AIF}(\psi,F_{\theta},1).\nonumber
\end{eqnarray}

For $p>1$, using~\eqref{eq:pop1}, we have
\begin{eqnarray}
\text{AIF}(\psi,\mathbf{x},p)&=&\frac{\left(\frac{1}{N}\sum\limits_{n=1}^N\left|\psi^{'}(x_n/T_N)\right|^{\frac{p}{p-1}}\right)^{\frac{p-1}{p}}}{\left|\frac{1}{N}\sum\limits_{n=1}^N x_n/T_N\psi^{'}(x_n/T_N)\right|}\nonumber \\&\overset{a.s.}{\rightarrow}&\frac{\left(\mathbb{E}_{\theta}\left[\left|\psi^{'}(X/\theta)\right|^{\frac{p}{p-1}}\right]\right)^{\frac{p-1}{p}}}{\left|\mathbb{E}_{\theta}\left[X/\theta\psi^{'}(X/\theta)\right]\right|}\nonumber\\&:=&\text{AIF}(\psi,F_{\theta},p).\nonumber
\end{eqnarray}
Since in scale model $F_{\theta}(x)=F_1(x/\theta)$, we have $f_{\theta}(x)=f_1\left(\frac{x}{\theta}\right)\frac{1}{\theta}$, and hence
\begin{eqnarray}
\text{AIF}(\psi,F_{\theta},p)=\frac{\left(\mathbb{E}_{F_1}\left[\left|\psi^{'}(X)\right|^{\frac{p}{p-1}}\right]\right)^{\frac{p-1}{p}}}{\left|\mathbb{E}_{F_1}\left[X\psi^{'}(X)\right]\right|}:=\text{AIF}(\psi,F_{1},p).\nonumber
\end{eqnarray}

For $p=2$, we have
\begin{eqnarray}
\text{AIF}(\psi,F_{1},2)=\frac{\left(\mathbb{E}_{F_1}\left[\psi^{'}(X)^{2}\right]\right)^{\frac{1}{2}}}{\left|\mathbb{E}_{F_1}\left[X\psi^{'}(X)\right]\right|}.\label{eq:AIFscale}
\end{eqnarray} 

\subsubsection{Minimizing $AIF(\psi, F_{\theta},p)$}\label{sec:noif}
In the following, among Fisher consistent estimators, we aim to design $\psi^{'}$ that minimizes $\text{AIF}(\psi,F_{1},2)$.

\begin{thm}\label{thm:scale}
	The optimal $\psi$ that minimizes $\text{AIF}(\psi, F_1, 2)$ has the following structure:
	\begin{itemize}
		\item For $x$ in the range of $f_1(x)$, $\psi^{'}$ satisfies 
		\begin{eqnarray}
		\psi^{'}(x)=\frac{x}{\mathbb{E}_{F_1}[X^2]}.\nonumber
		\end{eqnarray}
		\item $\psi(-\infty)$ is chosen as
		\begin{eqnarray}
		\psi(-\infty)=-\mathbb{E}_{F_{1}}\left[\int_{-\infty}^X\psi^{'}(t)\text{d}t\right].\nonumber
		\end{eqnarray}
	\end{itemize}

	With this choice of $\psi(x)$, the minimal value of $\text{AIF}(\psi, F_1, 2)$ is $1/\sqrt{\mathbb{E}_{F_1}[X^2]}$.
\end{thm}
\begin{proof}
	Please see Appendix~\ref{app:scale} for details.
\end{proof}

We note that for scale estimator $\frac{\partial \psi}{\partial \theta}=-\psi^{'}(x/\theta)x/\theta^2$, hence for this particular choice of $\psi^{'}$ in Theorem~\ref{thm:scale}, $\frac{\partial \psi}{\partial \theta}=-x^2/(\theta^3\mathbb{E}_{F_1}[X^2])$, which means $\psi(x,\theta)$ is monotone in $\theta$. This ensures that the obtained $\psi(x)$ satisfies the regularity conditions~\cite{Huber:Book:09} mentioned at the beginning of this section.

\subsubsection{Tradeoff between $AIF(\psi, F_{\theta},p)$ and $\gamma^*(\psi, F_{\theta})$}

Similar to the location estimation case, we can also design $\psi$ to minimize $\text{AIF}(\psi,F_{\theta},p)$ with a constraint on $\gamma^*(\psi, F_{\theta})$. From (2.3.17) of~\cite{Hampel:Book:86}, we know that for scale estimators
\begin{eqnarray}
\text{IF}(x,\psi, F_{\theta})=\frac{\psi(x/\theta)\theta}{\mathbb{E}_{F_{\theta}}[X/\theta\psi^{'}(X/\theta)]}.\nonumber
\end{eqnarray} 
To facilitate the analysis, we will focus on $\psi$ that is monotonic. Since in scale model, $\psi(x,\theta)=\psi(x/\theta)$, we can simply focus on the case of $\theta=1$. Hence, we will solve the following optimization problem to strike a desirable tradeoff between robustness against outliers and robustness against adversarial attacks.
\begin{eqnarray}
\min && \frac{\mathbb{E}_{F_1}\left[\psi^{'}(X)^{2}\right]}{\left(\mathbb{E}_{F_1}\left[X\psi^{'}(X)\right]\right)^2},\label{eq:scalecon}\\
\text{s.t.}&&	\gamma^{*}(\psi, F_1)=\sup\limits_{x}\left|\frac{\psi(x)}{\mathbb{E}_{F_{1}}[X\psi^{'}(X)]}\right|\leq \xi,\label{eq:ic}\\
&& \mathbb{E}_{F_1}[\psi]=0,\label{eq:fisher}\\
&& \psi^{'}(x)\geq 0.\label{eq:pos}
\end{eqnarray}
Here, constraint~\eqref{eq:ic} is a constraint on the outliers influence,~\eqref{eq:fisher} implies that $\psi$ is Fisher consistent. 

\begin{thm}\label{thm:scaletrade}
	The solution to~\eqref{eq:scalecon} has the following structure:
	
	\begin{itemize}
		\item $\psi^{'}$ has the following form
		\begin{eqnarray}
		&&\hspace{-12mm}	\psi^{'}(x)=\\\nonumber&&\hspace{-12mm}\left\{\begin{array}{cc}
		\nu^*x-\frac{ \vartheta_2^* +(\vartheta_1^*-\vartheta_2^*)F_1(x)}{f_1(x)}, & \hspace{-2mm}\nu^*x f_1(x)> \vartheta_2^* +(\vartheta_1^*-\vartheta_2^*)F_1(x);\\
		0,& \text{otherwise},
		\end{array}\right.\label{eq:case1}
		\end{eqnarray}
		in which $\nu^*$, $\vartheta_1^*\geq 0$ and $\vartheta_2^*\geq 0$ are chosen to satisfy
		\begin{eqnarray}
		\mathbb{E}_{F_{1}}[X\psi^{'}(X)]&=& 1,\nonumber\\
		\vartheta^*_1\left(		\int_{-\infty}^{\infty}\psi^{'}(x)F_1(x)\text{d}x- \xi\right)&=&0,\nonumber\\
		\vartheta^*_2\left(		\mathbb{E}_{F_1}\left[\int_{-\infty}^{X}\psi^{'}(t)dt\right]-\xi\right)&=&0,\nonumber
		\end{eqnarray}
		along with $\int_{-\infty}^{\infty}\psi^{'}(x)F_1(x)\text{d}x\leq  \xi$ and $	\mathbb{E}_{F_1}\left[\int_{-\infty}^{X}\psi^{'}(t)dt\right]\leq \xi$.
		
		\item $\psi(-\infty)$ is set to be $-\mathbb{E}_{F_1}\left[\int_{-\infty}^{X}\psi^{'}(t)dt\right]$.
	\end{itemize}
	
\end{thm}

\begin{proof}
	The proof follows similar strategy as that of the proof of Theorem~\ref{thm:tradeoff} and~\ref{thm:scale}. Details can be found in Appendix~\ref{app:scaletrade}.
\end{proof}
Similar to the location estimator case, the condition $\nu^*xf_1(x)> \vartheta_1^* F_1(x)+\vartheta_2^* (1-F_1(x))$ will limit the influences of data points at the tails.

\section{Extension: $L$-estimator}\label{sec:L}
In this section, we briefly discuss how to extend the analysis above to other class of estimators. We will use $L$-estimator as an example. $L$-estimator has the following form~\cite{Huber:Book:09,Hampel:Book:86}:
\begin{eqnarray}
T_N(\mathbf{x})=\sum\limits_{n=1}^N a_nx_{(n)},\nonumber
\end{eqnarray}
where $x_{(1)}\leq \cdots\leq x_{(N)}$ are the ordered sequence of $\mathbf{x}$, and $a_n$'s are coefficients. For example, for location estimator, a natural choice of $a_n$ is
\begin{eqnarray}
a_n=\frac{\int_{(n-1)/N}^{n/N}h(t)\text{d}t}{\int_{0}^{1}h(t)\text{d}t},\label{eq:location}
\end{eqnarray}
for a given function $h(t)$ such that $\int_{0}^{1}h(t)\text{d}t\neq 0$. For example, setting $h(t)=\delta(t-1/2)$ leads to the median estimator.

We first look at the given sample scenario. Let $\tilde{\mathbf{x}}=\mathbf{x}+\Delta \mathbf{x}$, and let $\tilde{x}_{(1)}\leq\cdots\leq \tilde{x}_{(N)}$ be the ordered sequence of $\tilde{\mathbf{x}}$. Hence, 
\begin{eqnarray}\label{eq:y}
T_N(\mathbf{x}+\Delta \mathbf{x})=\sum\limits_{n=1}^N a_n\tilde{x}_{(n)}.
\end{eqnarray}


For general $\Delta\mathbf{x}$, the ordering of $\mathbf{x}+\Delta\mathbf{x}$ may not necessarily be the same as the ordering of $\mathbf{x}$. For example, $\tilde{x}_{(1)}$ might come from $x_{(2)}$, i.e., $\tilde{x}_{(1)}=x_{(2)}+\Delta (x_{(2)})$. This possibility could make the following analysis messy. However, it is easy to see that when $\eta$ is sufficiently small (more specifically, when $N^{1/p}\eta\leq 1/2 \min\limits_{{x}_i\neq {x}_j }|{x}_i- {x}_j|$), the ordering of $\mathbf{x}+\Delta\mathbf{x}$ be the same as $\mathbf{x}$ for all $\Delta\mathbf{x}$'s that satisfy the constraint~\eqref{eq:norconstranit}. As the result, for the purpose of charactering AIF (which involves making $\eta\downarrow 0$), we can limit~\eqref{eq:y} to the following form
\begin{eqnarray}
T_N(\mathbf{x}+\Delta \mathbf{x})=\sum\limits_{n=1}^N a_n(x_{(n)}+\Delta (x_{(n)})).\nonumber
\end{eqnarray}

Hence
\begin{eqnarray}
T_N(\mathbf{x}+\Delta \mathbf{x})-T_N(\mathbf{x})=\sum\limits_{n=1}^{N}a_n\Delta x_{(n)},\nonumber
\end{eqnarray}
and~\eqref{eq:optimalintro} becomes
\begin{eqnarray}
\min&& -\sum\limits_{n=1}^{N}a_n\Delta x_{(n)},\nonumber\\
\text{s.t.}&& \frac{1}{N}||\Delta \mathbf{x}||_p^p\leq \eta^p.\nonumber
\end{eqnarray}

Using the exactly same approach as those in the proof of Theorem~\ref{thm:asc}, we have the following characterization.
For $p=1$, let $n^*=\arg\max\limits_{n} |a_n|$,
\begin{eqnarray}
\Delta x_{(n^*)}^*=\text{sign}\left\{ a_{n^*}\right\}N\eta,\nonumber
\end{eqnarray}
and $\Delta x_{(n)}^*=0,\forall n\neq n^*$. Hence,
\begin{eqnarray}
\text{AIF}(T_N,\mathbf{x},1)=N\left| a_{n^*}\right|.\nonumber
\end{eqnarray}

For $p>1$, we have
\begin{eqnarray}
\Delta x_{(n)}^*=\frac{|a_n|^{1/(p-1)}(N)^{1/p}}{(\sum|a_n|^{p/(p-1)})^{1/p}}\text{sign}(a_n)\eta.\nonumber
\end{eqnarray}

Hence, 
\begin{eqnarray}
\text{AIF}(\psi,\mathbf{x},p)&=&\sum a_n\frac{|a_n|^{1/(p-1)}(N)^{1/p}}{(\sum|a_n|^{p/(p-1)})^{1/p}}\text{sign}(a_n)\nonumber\\
&=&\frac{\sum\limits_{n=1}^N|a_n|^{p/(p-1)}}{\left(\frac{1}{N}\sum\limits_{n=1}^N|a_n|^{p/(p-1)}\right)^{1/p}}\label{eq:AIFLp}.
\end{eqnarray}


When $p=2$, this can be simplified to 
\begin{eqnarray}
\text{AIF}(T_N,\mathbf{x},2)&=&\frac{\sqrt{N}\sum\limits_{n=1}^N a_n^2}{\sqrt{\sum\limits_{n=1}^N  a_n^2}}=\sqrt{N\sum\limits_{n=1}^N  a_n^2}\label{eq:AIFL}.
\end{eqnarray}

For example, for $\alpha$-trimmed estimator~\cite{Hampel:Book:86} defined by
\begin{eqnarray}
T_N^{\alpha}(\mathbf{x})=\frac{1}{N-2\lfloor\alpha N\rfloor }\sum_{n=\lfloor\alpha N\rfloor+1}^{N-\lfloor\alpha N\rfloor}x_{(n)},\nonumber
\end{eqnarray} 
for a given parameter $0<\alpha<1/2$. For this $\alpha$-trimmed estimator, using~\eqref{eq:AIFLp}, we obtain
$$
\text{AIF}(T_N^{\alpha},\mathbf{x},p)=\frac{N^{1/p}}{(N-2\lfloor\alpha N\rfloor)^{1/p}}.
$$

If $a_n$s are chosen as~\eqref{eq:location}, then~\eqref{eq:AIFL} simplifies to
\begin{eqnarray}
\text{AIF}(T_N,\mathbf{x},2)&=&\sqrt{\frac{\frac{1}{N}\sum\limits_{n=1}^N\left(\int_{(n-1)/N}^{n/N}h(t)\text{d}t\right)^2}{\left(\frac{1}{N}\int_{0}^{1}h(t)\text{d}t\right)^2}}\nonumber\\
&\geq&\sqrt{\frac{\left(\frac{1}{N}\sum\limits_{n=1}^N\int_{(n-1)/N}^{n/N}h(t)\text{d}t\right)^2}{\left(\frac{1}{N}\int_{0}^{1}h(t)\text{d}t\right)^2}}\nonumber\\
&\geq&1,\nonumber
\end{eqnarray}
in which the first inequality is due to Jensen's inequality, and both inequalities become equality when $a_n=\int_{(n-1)/N}^{n/N}h(t)\text{d}t$ is a constant in $n$, \textcolor{black}{i.e., $a_n=1/N$ and the estimator becomes the empirical mean}.

\section{Numerical Examples}\label{sec:example}

In this section, we provide numerical examples to illustrate results obtained. 

We consider location estimation and illustrate the optimal estimator obtained in Theorem~\ref{thm:tradeoff} for the case when $f_0$ is exponential random variable $f_0(x)=e^{-x}, x\geq 0$, hence $f_{\theta}$ is shifted exponential random variable $f_{\theta}=e^{-(x-\theta)},x\geq \theta$ and the goal is to estimate $\theta$. As the exponential random variable has a unbounded support, choosing $\psi^{'}$ to be a constant, which minimizes AIF, will lead to an infinite IF. Hence, we use Theorem~\ref{thm:tradeoff} to characterize the optimal $\psi$ that minimizes AIF while satisfying the condition that $\text{IF}\leq \xi$.

For this particular class of distribution, the condition $\nu^* f_0(x)> \vartheta_1^* F_0(x)+\vartheta_2^*(1-F_0(x))$ becomes $0\leq x< a$ with the parameter $a$ chosen as
\begin{eqnarray}e^{-a}=\frac{\vartheta_1^*}{\nu^*+\vartheta_1^*-\vartheta_2^*}.\label{eq:a}\end{eqnarray}
Hence we have
\begin{eqnarray}
\psi^{'}(x)=\left\{\begin{array}{cc}
\nu^*+\vartheta_1^*-\vartheta_2^*-\vartheta_1^*e^{x}, & 0\leq x<a;\\
0,& \text{otherwise},
\end{array}\right.\nonumber
\end{eqnarray}
for which the parameters $\nu^*, \vartheta_1^*,\vartheta_2^*$ are chosen to satisfy the conditions specified in Theorem~\ref{thm:tradeoff}. After tedious calculation, conditions~\eqref{eq:1} - \eqref{eq:3} can be simplified to
\begin{eqnarray}
&&\hspace{-9mm}(\nu^*+\vartheta_1^*-\vartheta_2^*)(1-e^{-a})-\vartheta_1^*a=1,\nonumber\\
&&\hspace{-9mm}\vartheta_1^*((\nu^*+\vartheta_1^*-\vartheta_2^*)(a-1+e^{-a})-\vartheta_1^*(e^a-1)+a\vartheta_1^*-\xi)=0,\nonumber\\
&&\hspace{-9mm}\vartheta_2^*((\nu^*+\vartheta_1^*-\vartheta_2^*)(1-e^{-a})-\vartheta_1^*a-\xi)=0.\nonumber
\end{eqnarray}
From here, we know that if $\xi>1$, $\vartheta_2^*=0$, using this fact along with~\eqref{eq:a}, we have that the conditions are simplified to 
\begin{eqnarray}
\nu^*-a\vartheta_1^*&=&1,\nonumber\\
2a\vartheta_1^*+(a-2)\nu^*&=&\xi,\nonumber\\
\vartheta_2^*&=&0.\nonumber
\end{eqnarray}
Using these, we can express $\nu^*$ and $\vartheta_1^*$ in terms of $a$:
\begin{eqnarray}
\nu^*&=&\frac{\xi+2}{a},\nonumber\\
\vartheta_1^*&=&\frac{\xi+2-a}{a^2}.\nonumber
\end{eqnarray}

Finally, for any given $\xi>1$, the value of $a$ can be determined by~\eqref{eq:a}, which is simplified to
\begin{eqnarray}
e^{-a}=\frac{\vartheta_1^*}{\nu^*+\vartheta_1^*-\vartheta_2^*}=\frac{\xi+2-a}{(\xi+1)a+\xi+2}.\label{eq:asolution}
\end{eqnarray}

\begin{figure}[!htb]
	\centering
	\includegraphics[width=0.55\textwidth]{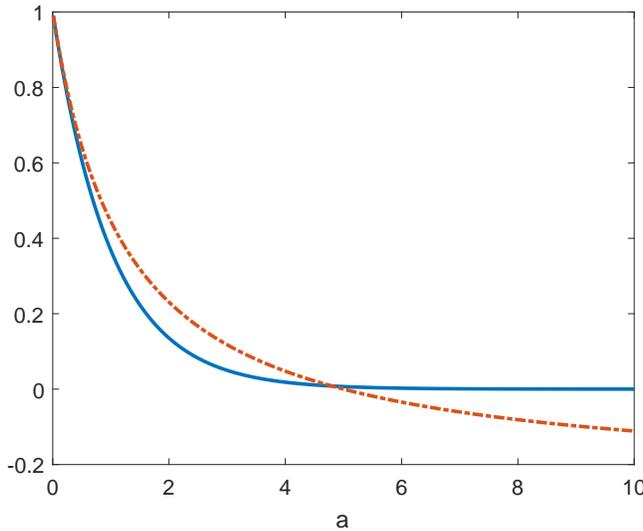}
	\caption{The solution of $a$}
	\label{fig:xi3}	
\end{figure}

It is easy to check that, for any given $\xi>1$, there is always a unique positive solution to~\eqref{eq:asolution}. For example, Figure~\ref{fig:xi3} illustrates the solution for $a$ when $\xi=3$. In this figure, the dotted curve is the right side of~\eqref{eq:asolution} and the solid curve is the left side of~\eqref{eq:asolution}. From the figure, we know that these two curves have two intersections $a=0$ and $a=4.8$. With these parameters, we know that
\begin{eqnarray}
\psi^{'}(x)=\left\{\begin{array}{cc}
1.0417-0.0087e^{x}, & 0\leq x\leq 4.8;\\
0,& \text{otherwise},
\end{array}\right.
\end{eqnarray}
hence the optimal $\psi$ is
\begin{eqnarray}
\psi^{'}(x)=\left\{\begin{array}{cc}
\xi,  &\hspace{-2mm} x\geq 4.8;\\
1.0417x-0.0087(e^{x}-1)-1, & \hspace{-2mm} 0\leq x\leq 4.8.
\end{array}\right.\nonumber
\end{eqnarray}

Figure~\ref{fig:psixi3} illustrates the obtained $\psi(x)$ for the case with $\xi=3$.
\begin{figure}[!htb]
	\centering
	\includegraphics[width=0.55\textwidth]{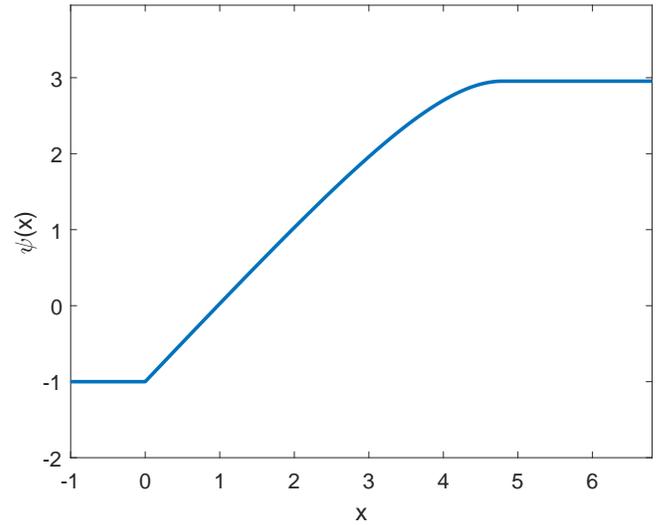}
	\caption{$\psi$ that minimizes AIF when $\text{IF}\leq 3$.}
	\label{fig:psixi3}	
\end{figure}

Figure~\ref{fig:AIFvsxi} illustrates the tradeoff curve between AIF and IF. We obtain this curve by solving~\eqref{eq:asolution} and other parameters using different values of $\xi$. As we can see from the curve, as $\xi$ increases, AIF decreases. Furthermore, the value of AIF converges to 1, the lower bound established in Section~\ref{sec:minAIF}.
\begin{figure}[!htb]
	\centering
	\includegraphics[width=0.55\textwidth]{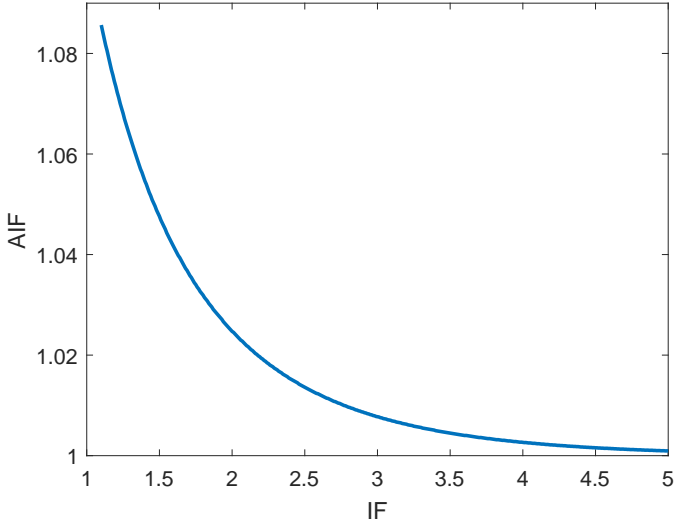}
	\caption{ Tradeoff between AIF and IF of location estimator for exponential random variables. }
	\label{fig:AIFvsxi}	
\end{figure}

\section{Conclusion}\label{sec:con}

Motivated by recent data analytics applications, we have studied adversarial robustness of robust estimators. We have introduced the concept of AIF to quantify an estimator's sensitivity to such adversarial attacks and have provided an approach to characterize AIF for given robust estimator. We have further designed optimal estimators that minimize AIF. From this characterization, we have identified a tradeoff between AIF and IF, and have designed estimators that strike a desirable tradeoff between these two quantities. 
\textcolor{black}{We note that AIF only captures the impact of vanishingly small corruptions. It is of interest to investigate the impact of non-vanishing corruptions and its connection with AIF in the future. }

\appendices
\section{Proof of Theorem~\ref{thm:asc}}\label{app:asc}
From~\eqref{eq:Mestimate}, we know that $T_N$ and $\mathbf{x}$ satisfy
\begin{eqnarray}
\sum\limits_{n=1}^N \psi(x_n,T_N)=0.\nonumber
\end{eqnarray}

Hence, we have
\begin{eqnarray}\label{eq:partialxn}
\frac{\partial }{\partial x_n}T_N=\frac{-\frac{\partial }{\partial x}[\psi]_{x=x_n, \theta=T_N}}{\sum\limits_{n=1}^{N}\frac{\partial }{\partial \theta}[\psi]_{x=x_n,\theta=T_N}}.
\end{eqnarray}

Based on Taylor expansion, we have
\begin{eqnarray}
&&T_N(\mathbf{x}+\Delta \mathbf{x})-T_N(\mathbf{x})\nonumber\\&&\hspace{5mm}=\sum\limits_{n=1}^{N}\Delta x_n \frac{\partial }{\partial x_n}T_N+\text{higher order terms}. \nonumber
\end{eqnarray}

When $\eta$ is small, the adversary can solve the following problem and obtain an $o(\eta)$ optimal solution
\begin{eqnarray}\label{eq:linear}
\min\limits_{\Delta\mathbf{x}}&& -\sum\limits_{n=1}^{N}\Delta x_n c_n,\nonumber\\
\text{s.t.}&& ||\Delta \mathbf{x}||_p^p\leq N\eta^p,
\end{eqnarray}
in which 
\begin{eqnarray}
c_n:=\frac{\partial }{\partial x_n}T_N.\nonumber
\end{eqnarray}

For $p=1$, this is a linear programing problem, whose solution is simple. In particular, let $n^*=\arg\max\limits_{n} | \frac{\partial }{\partial x_n}T_N|$, which is the same as $\arg\max\limits_{n} \left |\frac{\partial }{\partial x}[\psi]_{x=x_n, \theta=T_N}\right|$ due to~\eqref{eq:partialxn}, it is easy to check that we have
\begin{eqnarray}
\Delta x_{n^*}^*=\text{sign}\left\{ \frac{\partial }{\partial x_{n^*}}T_N\right\}N\eta,\nonumber
\end{eqnarray}
and $\Delta x_{n}^*=0,\forall n\neq n^*$. Hence,
\begin{eqnarray}
\text{AIF}(\psi,\mathbf{x},1)=N\left| \frac{\partial }{\partial x_{n^*}}T_N\right|=\frac{\left|N\frac{\partial }{\partial x}[\psi]_{x=x_{n^*}, \theta=T_N}\right|}{\left|\sum\limits_{n=1}^{N}\frac{\partial }{\partial \theta}[\psi]_{x=x_n,\theta=T_N}\right|}.\nonumber
\end{eqnarray}

For \textcolor{black}{ $\infty>p> 1$},~\eqref{eq:linear} is a convex optimization problem. To solve this, we form Lagrange
\begin{eqnarray}
\mathcal{L}(\Delta \mathbf{x},\lambda)= -\sum\limits_{n=1}^{N}\Delta x_n c_n+ \lambda\left(||\Delta \mathbf{x}||_p^p-N\eta^p\right).\nonumber
\end{eqnarray}

The corresponding optimality conditions are:
\begin{eqnarray}
-c_n+\lambda^* p \text{sign} (\Delta x_n^*) |\Delta x_n^*|^{p-1}&=&0,  \forall n\label{eq:lag}\\
\lambda^*&\geq&0,\nonumber\\
\lambda^* (||\Delta \mathbf{x}^*||_p^p-N\eta^p)&=&0.\nonumber
\end{eqnarray}

From~\eqref{eq:lag}, we know that $\lambda^*\neq 0$, hence 
\begin{eqnarray}
||\Delta \mathbf{x}^*||_p^p=N\eta^p,\label{eq:equal}
\end{eqnarray}
and
\begin{eqnarray}
\text{sign} (\Delta x_n^*) |\Delta x_n^*|^{p-1}=\frac{c_n}{\lambda^* p}.\label{eq:sign}
\end{eqnarray}

From~\eqref{eq:sign} and the fact that $\lambda^* p$ is positive, we know $\text{sign}(\Delta x_n^*)=\text{sign}(c_n)$, and hence we have
\begin{eqnarray}
|\Delta x_n^*|^{p-1}=\frac{|c_n|}{\lambda^* p},\nonumber
\end{eqnarray}
which can be simplified further to
\begin{eqnarray}
\Delta x_n^*=\left(\frac{|c_n|}{\lambda^* p}\right)^{1/(p-1)} \text{sign} (c_n).\nonumber
\end{eqnarray}

Combining these with~\eqref{eq:equal}, we obtain the value of $\lambda^*$:
\begin{eqnarray}
\lambda^*=\frac{1}{p}\left(\frac{\sum\limits_{n=1}^N |c_n|^{p/(p-1)}}{N\eta^p}\right)^{(p-1)/p}.\nonumber
\end{eqnarray}

As the result, we have
\begin{eqnarray}
\Delta x_n^*=\frac{|c_n|^{1/(p-1)}(N)^{1/p}}{(\sum|c_n|^{p/(p-1)})^{1/p}}\text{sign}(c_n)\eta.\nonumber
\end{eqnarray}

Hence, 
\begin{eqnarray}
\text{AIF}(\psi,\mathbf{x},p)&=&\sum c_n\frac{|c_n|^{1/(p-1)}(N)^{1/p}}{(\sum|c_n|^{p/(p-1)})^{1/p}}\text{sign}(c_n)\nonumber\\
&=&\frac{\sum\limits_{n=1}^N|c_n|^{p/(p-1)}}{\left(\frac{1}{N}\sum|c_n|^{p/(p-1)}\right)^{1/p}}.\nonumber
\end{eqnarray}

Using~\eqref{eq:partialxn}, we can further simplify the expression to
\begin{eqnarray}
\text{AIF}(\psi,\mathbf{x},p)=\frac{\left(\frac{1}{N}\sum\limits_{n=1}^N\left|\frac{\partial }{\partial x}[\psi]_{x=x_n,\theta=T_N}\right|^{\frac{p}{p-1}}\right)^{\frac{p-1}{p}}}{\left|\frac{1}{N}\sum\limits_{n=1}^N \frac{\partial}{\partial \theta}[\psi]_{x=x_n,\theta=T_N}\right|}.\label{eq:AIFpapp}
\end{eqnarray}

\textcolor{black}{For $p=\infty$, as $N^{1/p}\overset{p\rightarrow \infty}\rightarrow 1 $, \eqref{eq:linear} can be written as \begin{eqnarray}\label{eq:lineari}
	\min\limits_{\Delta\mathbf{x}}&& -\sum\limits_{n=1}^{N}\Delta x_n c_n,\nonumber\\
	\text{s.t.}&& ||\Delta \mathbf{x}||_{\infty}\leq \eta.
	\end{eqnarray}
	It is easy to see that the optimal $\Delta x_{n}^*=\eta \text{sign}\{c_n\} $. Hence,
	\begin{eqnarray}
\text{AIF}(\psi,\mathbf{x},p)=\sum\limits_{n=1}^N |c_n|=\frac{\sum\limits_{n=1}^N\left|\frac{\partial }{\partial x}[\psi]_{x=x_n, \theta=T_N}\right|}{\left|\sum\limits_{n=1}^{N}\frac{\partial }{\partial \theta}[\psi]_{x=x_n,\theta=T_N}\right|},
	\end{eqnarray}
	which is the limit of~\eqref{eq:AIFpapp} as $p\rightarrow \infty$. 
}
\section{Proof of Theorem~\ref{thm:tradeoff}}\label{app:tradeoff}
As $\psi^{'}(x)\geq 0$, we have $\mathbb{E}_{F_{0}}[\psi^{'}(X)]>0$, and $\sup\limits_{x}|\psi(x)|$ is either $\psi(\infty)$ or $-\psi(-\infty)$. Hence for $p=2$, the optimization problem~\eqref{eq:opt} is equivalent to
\begin{eqnarray}
\min&&
\frac{\mathbb{E}_{F_0}[\psi^{'}(X)^2]}{(\mathbb{E}_{F_0}[\psi^{'}(X)])^2}\nonumber\\
\text{s.t.}&& \frac{\psi(-\infty)+\int_{-\infty}^{\infty}\psi^{'}(x)\text{d}x}{\mathbb{E}_{F_{0}}[\psi^{'}(X)]}
\leq \xi,\nonumber\\
&&\frac{-\psi(-\infty)}{\mathbb{E}_{F_{0}}[\psi^{'}(X)]}
\leq \xi,\nonumber\\
&&	\psi(-\infty)+\mathbb{E}_{F_{0}}\left[\int_{-\infty}^X\psi^{'}(t)\text{d}t\right]=0,\nonumber\\
&& \psi^{'}\geq 0.\nonumber
\end{eqnarray}
As the objective function does not involve $\psi(-\infty)$, we can first solve
\begin{eqnarray}
\min&&
\frac{\mathbb{E}_{F_{0}}[\psi^{'}(X)^2]}{(\mathbb{E}_{F_{0}}[\psi^{'}(X)])^2},\nonumber\\
\text{s.t.}&& \frac{-\mathbb{E}_{F_{0}}\left[\int_{-\infty}^X\psi^{'}(t)\text{d}t\right]+\int_{-\infty}^{\infty}\psi^{'}(x)\text{d}x}{\mathbb{E}_{F_{0}}[\psi^{'}(X)]}
\leq \xi,\nonumber\\
&& \frac{\mathbb{E}_{F_{0}}\left[\int_{-\infty}^X\psi^{'}(t)\text{d}t\right]}{\mathbb{E}_{F_{0}}[\psi^{'}(X)]}
\leq \xi,\nonumber\\
&& \psi^{'}(x)\geq 0.\nonumber
\end{eqnarray}
After obtaining the solution, we can simply set $\psi(-\infty)=-\mathbb{E}_{F_{0}}\left[\int_{-\infty}^X\psi^{'}(t)\text{d}t\right]$ to make $\psi$ Fisher consistent.

To simplify the notation, in the remainder of the proof, we will use $g(x)$ to denote $\psi^{'}(x)$. We now further simplify the optimization problem. First, we have
\begin{eqnarray}
\mathbb{E}_{F_{0}}\left[\int_{-\infty}^Xg(t)\text{d}t\right]&=&\int_{-\infty}^{\infty}f_0(x)\left[\int_{-\infty}^xg(t)\text{d}t\right]\text{d}x\nonumber\\
&=&\int_{-\infty}^{\infty}g(t)\left[\int_{t}^{\infty}f_0(x)\text{d}x\right]\text{d}t\nonumber\\
&=&\int_{-\infty}^{\infty}g(t)\left[1-F_0(t)\right]\text{d}t.\label{eq:int}
\end{eqnarray}
Coupled with the fact that $g(x)\geq 0$ and $f_0(x)\geq 0$, the optimization above is equivalent to	
\begin{eqnarray}
\min&&
\frac{\int_{-\infty}^{\infty}g^2(x)f_0(x)\text{d}x}{\left(\int_{-\infty}^{\infty}g(x)f_0(x)\text{d}x\right)^2},\nonumber\\
\text{s.t.}&& \int_{-\infty}^{\infty}g(x)F_0(x)\text{d}x
\leq \xi \int_{-\infty}^{\infty}g(x)f_0(x)\text{d}x,\nonumber\\
&& \int_{-\infty}^{\infty}g(x)\left[1-F_0(x)\right]\text{d}x
\leq \xi \int_{-\infty}^{\infty}g(x)f_0(x)\text{d}x,\nonumber\\
&& g(x)\geq 0.\nonumber
\end{eqnarray}

It is clear that the optimization problem is scale invariant in the sense that if $g^*(x)$ is a solution to this problem, then for any positive constant $c$,  $cg^*(x)$ is also a solution to this problem. As a result, without loss of generality, we can assume $\int_{-\infty}^{\infty}g(x)f_0(x)\text{d}x=1$. Using this, we can further simplify the optimization problem to
\begin{eqnarray}
\min&&
\frac{1}{2}	\int_{-\infty}^{\infty}g^2(x)f_0(x)\text{d}x,\nonumber\\
\text{s.t.}	&&\int_{-\infty}^{\infty}g(x)f_0(x)\text{d}x=1,\nonumber\\
&& \int_{-\infty}^{\infty}g(x)F_0(x)\text{d}x
\leq \xi ,\nonumber\\
&& \int_{-\infty}^{\infty}g(x)\left[1-F_0(x)\right]\text{d}x
\leq \xi,\nonumber\\
&& g(x)\geq 0.\nonumber
\end{eqnarray}

To solve this convex functional minimization problem, we first form the Lagrangian function
\begin{eqnarray}
\mathcal{L}&=&\frac{1}{2}\int_{-\infty}^{\infty}g^2(x)f_0(x)\text{d}x+\nu\left(-\int_{-\infty}^{\infty}g(x)f_0(x)\text{d}x+1\right)\nonumber\\
&&-\lambda(x)g(x)+\vartheta_1\left(\int_{-\infty}^{\infty}g(x)F_0(x)\text{d}x- \xi \right)
\nonumber\\
&&+\vartheta_2\left(\int_{-\infty}^{\infty}g(x)\left[1-F_0(x)\right]\text{d}x- \xi\right).\nonumber
\end{eqnarray}

Let $H=\frac{1}{2}g^2(x)f_0(x)-\nu g(x)f_0(x)+\vartheta_1g(x)F_0(x)+\vartheta_2g(x)(1-F_0(x))-\lambda(x)g(x)$. As no derivative $g^{'}(x)$ is involved in $H$, the optimality condition Euler-Lagrange equation~\cite{Kot:Book:14}
\begin{eqnarray}
\frac{\partial H}{\partial g}-\frac{d}{d x}\left(\frac{\partial H}{\partial g^{'}}\right)=0\nonumber
\end{eqnarray}
simplifies to
\begin{eqnarray}
g^*(x)f_0(x)-\nu^* f_0(x)+\vartheta_2^*+(\vartheta_1^*-\vartheta_2^*)F_0(x)-\lambda^*(x)=0,\label{eq:gf}
\end{eqnarray}
in which the parameters $\vartheta_1^*\geq 0$, $\vartheta^*_2\geq 0$, $\lambda^*(x)\geq 0$ satisfy~\cite{Burger:Book:03}
\begin{eqnarray}
&&\int_{-\infty}^{\infty}g^*(x)f_0(x)\text{d}x=1,\nonumber\\
&&\vartheta^*_1\left( \int_{-\infty}^{\infty}g^*(x)F_0(x)\text{d}x
- \xi \right )=0,\nonumber\\
&& \vartheta^*_2\left(\int_{-\infty}^{\infty}g^*(x)\left[1-F_0(x)\right]\text{d}x- \xi\right)=0,\nonumber\\
&& \lambda^*(x)g(x)\geq 0.\label{eq:posi}
\end{eqnarray}

From~\eqref{eq:gf}, for $x$ in the range of $f_0(x)$, we have
\begin{eqnarray}
g^*(x)=\frac{\lambda^*(x)+\nu^* f_0(x)-\vartheta^*_2-(\vartheta^*_1-\vartheta^*_2)F_0(x)}{f_0(x)}.\nonumber
\end{eqnarray}

Combining this with the condition~\eqref{eq:posi}, we know that if $\nu^* f_0(x)-\vartheta^*_2-(\vartheta^*_1-\vartheta^*_2)F_0(x)>0$, then $\lambda^*(x)=0$. On the other hand, if $\nu^* f_0(x)-\vartheta^*_2-(\vartheta^*_1-\vartheta^*_2)F_0(x)<0$, then $g^*(x)=0$. As a result, we have

\begin{eqnarray}
&&\hspace{-8mm}g^*(x)=\nonumber\\
&&\hspace{-10mm}\left\{\begin{array}{cc}
\nu^*-\frac{ \vartheta_2^*+(\vartheta_1^*-\vartheta_2^*)F_0(x)}{ f_0(x)}, & \nu^* f_0(x)> \vartheta_2^*+(\vartheta_1^*-\vartheta_2^*)F_0(x);\\
0,& \text{otherwise},
\end{array}\right.\nonumber
\end{eqnarray}
which completes the proof.
\section{Proof of Theorem~\ref{thm:scale}}\label{app:scale}
First of all, minimizing~\eqref{eq:AIFscale} is same as solving
\begin{eqnarray}
\hspace{-3mm}\min &&\hspace{-6mm} \frac{\mathbb{E}_{F_1}\left[\psi^{'}(X)^{2}\right]}{\left(\mathbb{E}_{F_1}\left[X\psi^{'}(X)\right]\right)^2},\label{eq:cost}	\\
\hspace{-3mm}\text{s.t.}&&\hspace{-6mm} \mathbb{E}_{F_1}[\psi(X)]=\psi(-\infty)+\mathbb{E}_{F_{1}}\left[\int_{-\infty}^X\psi^{'}(t)\text{d}t\right]=0,\label{eq:constraint}
\end{eqnarray}
in which the condition $\mathbb{E}_{F_1}[\psi(X)]=0$ ensures that the estimator is Fisher consistent.

As $\psi(-\infty)$ does not appear in the objective function, we can solve~\eqref{eq:cost} without the constraint~\eqref{eq:constraint} first. After that, we can simply set $$\psi(-\infty)=-\mathbb{E}_{F_{1}}\left[\int_{-\infty}^X\psi^{'}(t)\text{d}t\right]$$ so that the constraint~\eqref{eq:constraint} will be satisfied. Furthermore, similar to the proof of Theorem~\ref{thm:tradeoff}, to simplify the notation, we will use $g(x)$ to denote $\psi^{'}(x)$. It is clear from~\eqref{eq:cost} that the cost function is scale-invariant. Hence, without loss of generality, we can assume $\left(\mathbb{E}_{F_1}\left[Xg(X)\right]\right)^2=1$, for which we can further focus on $\mathbb{E}_{F_1}\left[Xg(X)\right]=1$. Combining all these together, the optimization problem can be converted to 
\begin{eqnarray}
\min &&\frac{1}{2}\int_{-\infty}^{\infty}g^2(x)f_1(x)\text{d}x,\nonumber\\
\text{s.t.}&& \int_{-\infty}^{\infty}xg(x)f_1(x)\text{d}x=1.\nonumber
\end{eqnarray}
For this convex calculus of variations problem, we form Lagrange function
\begin{eqnarray}
\mathcal{L}=\int_{-\infty}^{\infty}\frac{1}{2}g^2(x)f_1(x)\text{d}x+\nu\left(-\int_{-\infty}^{\infty}xg(x)f_1(x)\text{d}x-1\right).\nonumber
\end{eqnarray}
The corresponding Euler-Lagrange equation can be simplified to
\begin{eqnarray}
g^*(x)f_1(x)-\nu^* xf_1(x)=0,\label{eq:euler}
\end{eqnarray}
and the optimal value of $\nu^*$ is selected to satisfy the condition
\begin{eqnarray}
\int_{-\infty}^{\infty}xg^*(x)f_1(x)\text{d}x=1.\label{eq:con}
\end{eqnarray}
From~\eqref{eq:euler}, we know that in the range of $X$ where $f_1(x)>0$, $g^*(x)=\nu^* x$. Plugging this into~\eqref{eq:con}, we obtain 
\begin{eqnarray}
\nu^*=\frac{1}{\int_{-\infty}^{\infty}x^2f_1(x)\text{d}x}.\nonumber
\end{eqnarray}
As the result, for $x$ in the range of $f_1(x)$, the optimal $g^*(x)$ is 
\begin{eqnarray}
g^*(x)=\frac{x}{\mathbb{E}_{F_1}[X^2]},\nonumber
\end{eqnarray}
and 
$\psi(-\infty)=-\mathbb{E}_{F_{1}}\left[\int_{-\infty}^Xg(t)\text{d}t\right]$.
\section{Proof of Theorem~\ref{thm:scaletrade}}\label{app:scaletrade}
Following the same strategy as those in the proof of Theorem~\ref{thm:tradeoff}, we can first solve the following problem 
\begin{eqnarray}
\min && \frac{\mathbb{E}_{F_1}\left[\psi^{'}(X)^{2}\right]}{\left(\mathbb{E}_{F_1}\left[X\psi^{'}(X)\right]\right)^2},\nonumber\\
\text{s.t.}&&	-\mathbb{E}_{F_1}\left[\int_{-\infty}^{X}\psi^{'}(t)dt\right]+\int_{-\infty}^{\infty}\psi^{'}(t)\text{d}t\nonumber\\&&\hspace{8mm}\leq \xi\left|\mathbb{E}_{F_{1}}[X\psi^{'}(X)]\right|,\nonumber\\
&&	\mathbb{E}_{F_1}\left[\int_{-\infty}^{X}\psi^{'}(t)dt\right]\leq \xi\left|\mathbb{E}_{F_{1}}[X\psi^{'}(X)]\right|,\nonumber\\
&& \psi^{'}(x)\geq 0,\nonumber
\end{eqnarray}
and then set $\psi(-\infty)=-\mathbb{E}_{F_1}\left[\int_{-\infty}^{X}\psi^{'}(t)dt\right]$ to satisfy the Fisher consistent constraint~\eqref{eq:fisher}.

Now, we consider two different cases depending on whether $\mathbb{E}_{F_{1}}[X\psi^{'}(X)]$ is positive or negative. In the following, to simplify notation, we will use $g(x)$ to denote $\psi^{'}(x)$.

We will solve the case with $\mathbb{E}_{F_{1}}[Xg(X)]>0$ in detail. The case $\mathbb{E}_{F_{1}}[Xg(X)]<0$ can be solved in the similar manner. With $\mathbb{E}_{F_{1}}[Xg(X)]>0$, the optimization problem is same as

\begin{eqnarray}
\min && \frac{\mathbb{E}_{F_1}\left[g^{2}(X)\right]}{\left(\mathbb{E}_{F_1}\left[Xg(X)\right]\right)^2},\nonumber\\
\text{s.t.}&&	-\mathbb{E}_{F_1}\left[\int_{-\infty}^{X}g(t)dt\right]+\int_{-\infty}^{\infty}g(t)\text{d}t\leq \xi\mathbb{E}_{F_{1}}[Xg(X)],\nonumber\\
&&	\mathbb{E}_{F_1}\left[\int_{-\infty}^{X}g(t)dt\right]\leq \xi\mathbb{E}_{F_{1}}[Xg(X)],\nonumber\\
&& g(x)\geq 0.\nonumber
\end{eqnarray}

Similar to the optimization problems in Theorem~\ref{thm:tradeoff} and~\ref{thm:scale}, the optimization problem is scale-invariant, and hence without loss of generality, we can focus on $\mathbb{E}_{F_{1}}[Xg(X)]=1$. Furthermore, similar to~\eqref{eq:int}, we have $
\mathbb{E}_{F_{1}}\left[\int_{-\infty}^Xg(t)\text{d}t\right]=\int_{-\infty}^{\infty}g(t)\left[1-F_1(t)\right]\text{d}t.
$ The problem is then converted to
\begin{eqnarray}
\min && \int_{-\infty}^{\infty}g^{2}(x)f_1(x)\text{d}x,\nonumber\\
\text{s.t.}&& \int_{-\infty}^{\infty}xg(x)f_1(x)\text{d}x=1,\nonumber\\
&&	\int_{-\infty}^{\infty}g(x)F_1(x)\text{d}x\leq \xi,\nonumber\\
&&	\int_{-\infty}^{\infty}g(x)[1-F_1(x)]\text{d}x\leq \xi,\nonumber\\
&& g(x)\geq 0.\nonumber
\end{eqnarray}

To solve this convex functional minimization problem, we first form the Lagrangian function
\begin{eqnarray}
\mathcal{L}&=&\hspace{-3mm}\frac{1}{2}\int_{-\infty}^{\infty}g^2(x)f_1(x)\text{d}x+\nu\left(-\int_{-\infty}^{\infty}xg(x)f_1(x)\text{d}x+1\right)\nonumber\\
&&\hspace{-3mm}-\lambda(x)g(x)+\vartheta_1\left(\int_{-\infty}^{\infty}g(x)F_1(x)\text{d}x
- \xi \right)\nonumber\\
&&\hspace{-3mm}+\vartheta_2\left(\int_{-\infty}^{\infty}g(x)\left[1-F_1(x)\right]\text{d}x- \xi\right).\nonumber
\end{eqnarray}

Let $F=\frac{1}{2}g^2(x)f_1(x)-\nu x g(x)f_1(x)+\vartheta_1g(x)F_1(x)+\vartheta_2g(x)(1-F_1(x))-\lambda(x)g(x)$. As no derivative $g^{'}(x)$ is involved in $F$, the Euler-Lagrange equation
\begin{eqnarray}
\frac{\partial F}{\partial g}-\frac{d}{d x}\left(\frac{\partial F}{\partial g^{'}}\right)=0\nonumber
\end{eqnarray}
simplifies to
\begin{eqnarray}
&&g^*(x)f_1(x)-\nu^* x f_1(x)+\vartheta^*_2+(\vartheta^*_1-\vartheta^*_2)F_1(x)-\lambda^*(x)\nonumber\\&&\hspace{15mm}=0,\label{eq:gfs}
\end{eqnarray}
in which the parameters $\vartheta^*_1\geq 0$, $\vartheta^*_2\geq 0$, $\lambda^*(x)\geq 0$ satisfy
\begin{eqnarray}
&&\int_{-\infty}^{\infty}xg^*(x)f_1(x)\text{d}x=1,\nonumber\\
&&\vartheta^*_1\left( \int_{-\infty}^{\infty}g^*(x)F_1(x)\text{d}x
- \xi \right )=0,\nonumber\\
&& \vartheta^*_2\left(\int_{-\infty}^{\infty}g^*(x)\left[1-F_1(x)\right]\text{d}x- \xi\right)=0,\nonumber\\
&& \lambda^*(x)g^*(x)\geq 0.\label{eq:posis}
\end{eqnarray}

From~\eqref{eq:gfs}, for those $x$ with $f_1(x)>0$, we have
\begin{eqnarray}
g^*(x)=\frac{\lambda^*(x)+\nu^* x f_1(x)-\vartheta^*_2-(\vartheta^*_1-\vartheta^*_2)F_1(x)}{f_1(x)}.\nonumber
\end{eqnarray}

Combining this with the condition~\eqref{eq:posis}, we know that if $\nu^* x f_1(x)-\vartheta^*_2-(\vartheta^*_1-\vartheta^*_2)F_1(x)>0$, then $\lambda^*(x)=0$. On the other hand, if $\nu^* f_1(x)-\vartheta^*_2-(\vartheta^*_1-\vartheta^*_2)F_1(x)<0$, then $g^*(x)=0$. As the result, we have

\begin{eqnarray}
&&\hspace{-8mm}g^*(x)=\nonumber\\
&&\hspace{-8mm}\left\{\begin{array}{cc}
\nu^* x-\frac{ \vartheta_2^*+(\vartheta_1^*-\vartheta_2^*)F_1(x)}{ f_1(x)}, & \hspace{-2mm}\nu^* x f_1(x)> \vartheta_2^*+(\vartheta_1^*-\vartheta_2^*)F_1(x);\nonumber\\
0,& \text{otherwise}.
\end{array}\right.
\end{eqnarray}

\bibliographystyle{ieeetr}
\bibliography{macros,MachineLearning}

\begin{thebibliography}{10}

\bibitem{Huber:AMS:64}
P.~J. Huber, ``Robust estimation of a location parameter,'' {\em Ann. Math.
  Statist.}, vol.~35, pp.~73--101, 1964.

\bibitem{Huber:Book:09}
P.~Huber and E.~Ronchetti, {\em Robust statistics}.
\newblock Wiley, 2009.

\bibitem{Hampel:Book:86}
F.~Hampel, E.~Ronchetti, P.~Rousseeuw, and W.~Stahel, {\em Robust statistics:
  The Approach Based on Influence Functions}.
\newblock Wiley, 2009.

\bibitem{Bhatia:NIPS:15}
K.~Bhatia, P.~Jain, and P.~Kar, ``Robust regression via hard thresholding,'' in
  {\em Proc. Advances in Neural Information Processing Systems}, (Montreal,
  Canada), pp.~721--729, 2015.

\bibitem{Prasad:JRSSB:19}
A.~Prasad, A.~S. Suggala, S.~Balakrishnan, and P.~Ravikumar, ``Robust
  estimation via robust gradient estimation,'' {\em Journal of the Royal
  Statistical Society}.
\newblock Submitted.

\bibitem{Diakonikolas:FOCS:16}
I.~Diakonikolas, G.~Kamath, D.~M. Kane, J.~Li, A.~Moitra, and A.~Stewart,
  ``Robust estimators in high dimensions without the computational
  intractability,'' in {\em Proc. of IEEE Symposium on Foundations of Computer
  Science}, (New Brunswick, NJ), Oct. 2016.

\bibitem{Diakonikolas:ICML:2017}
I.~Diakonikolas, G.~Kamath, D.~M. Kane, J.~Li, A.~Moitra, and A.~Stewart,
  ``Being robust (in high dimensions) can be practical,'' in {\em Proc.
  International Conference on Machine Learning}, (Sydney, NSW, Australia),
  pp.~999--1008, 2017.

\bibitem{Balakrishnan:COLT:17}
S.~Balakrishnan, S.~S. Du, J.~Li, and A.~Singh, ``Computationally efficient
  robust sparse estimation in high dimensions,'' in {\em Proc. Conference on
  Learning Theory}, Proceedings of Machine Learning Research, (Amsterdam,
  Netherlands), pp.~169--212, July 2017.

\bibitem{Szegedy2013IntriguingPO}
C.~Szegedy, W.~Zaremba, I.~Sutskever, J.~Bruna, D.~Erhan, I.~J. Goodfellow, and
  R.~Fergus, ``Intriguing properties of neural networks,'' {\em CoRR},
  vol.~abs/1312.6199, 2013.

\bibitem{Goodfellow:ICLR:15}
I.~J. Goodfellow, J.~Shlens, and C.~Szegedy, ``Explaining and harnessing
  adversarial examples,'' in {\em Proc. International Conference on Learning
  Representations}, (San Diego, CA), May 2015.

\bibitem{Carlini:ISSP:17}
N.~Carlini and D.~Wagner, ``Towards evaluating the robustness of neural
  networks,'' in {\em Proc. IEEE Intl. Symposium on Security and Privacy}, (San
  Jose, CA), May 2017.

\bibitem{Mei:AAAI:15}
S.~Mei and X.~Zhu, ``Using machine teaching to identify optimal training-set
  attacks on machine learners,'' in {\em Proc. AAAI Conference on Artificial
  Intelligence}, (Austin, TX), Jan. 2015.

\bibitem{Hampel:PHD:68}
F.~Hampel, ``Contributions to the theory of robust estimation,'' {\em Ph.D.
  thesis, University of California, Berkeley}, 1968.

\bibitem{Burger:Book:03}
M.~Burger, {\em Infinite-dimensional Optimization and Optimal Design}.
\newblock 2003.
\newblock Lecture notes, available at
  "ftp://ftp.math.ucla.edu/pub/camreport/cam04-11.pdf".

\bibitem{Kot:Book:14}
M.~Kot, {\em A first course in the calculus of variations}.
\newblock Providence, Rhode Island: American Mathematical Society, 2014.

\bibitem{Pimentel:ISIT:17}
D.~L. Pimentel-Alarcon, A.~Biswas, and C.~R. Solis-Lemus, ``Adversarial
  principal component analysis,'' in {\em Proc. IEEE Intl. Symposium on Inform.
  Theory}, pp.~2363--2367, July 2017.

\bibitem{Candes:JACM:11}
E.~J. Candes, X.~Li, Y.~Ma, and J.~Wright, ``Robust principal component
  analysis?,'' {\em Journal of ACM}, vol.~58, no.~1, pp.~1--37, 2011.

\bibitem{Jagielski:ISSP:18}
M.~Jagielski, A.~Oprea, B.~Biggio, C.~Liu, C.~Nita-Rotaru, and B.~Li,
  ``Manipulating machine learning: {P}oisoning attacks and countermeasures for
  regression learning,'' in {\em Proc. IEEE Intl. Symposium on Security and
  Privacy}, (San Francisco, CA), May 2018.

\bibitem{Biggio:ICML:12}
B.~Biggio, B.~Nelson, and P.~Laskov, ``Poisoning attacks against support vector
  machines,'' in {\em Proc. International Conference on Machine Learning},
  (Edinburgh, Scotland), pp.~1467--1474, 2012.

\bibitem{Charikar:STOC:17}
M.~Charikar, J.~Steinhardt, and G.~Valiant, ``Learning from untrusted data,''
  in {\em Proc. Annual ACM SIGACT Symposium on Theory of Computing},
  pp.~47--60, 2017.

\bibitem{Suggala:COLT:19}
A.~S. Suggala, K.~Bhatia, P.~Ravikumar, and P.~Jain, ``Adaptive hard
  thresholding for near-optimal consistent robust regression,'' in {\em Proc.
  Conference on Learning Theory}, Proceedings of Machine Learning Research,
  (Phoenix, AZ), June 2019.

\bibitem{Bayraktar:SMDS:19}
E.~Bayraktar and L.~Lai, ``On the adversarial robustness of multivariate robust
  estimators,'' {\em SIAM Journal on Mathematics of Data Science}, Aug. 2019.
\newblock Submitted.

\bibitem{Ferguson:Book:96}
T.~Ferguson, {\em A course in large sample theory}.
\newblock Chapman and Hall, London, UK, 1996.

\bibitem{Croux:SPL:98}
C.~Croux, ``Limit behavior of the empirical influence function of the median,''
  {\em Statistics $\&$ Probability Letter}, vol.~37, pp.~331--340, 1998.

\end{thebibliography}

\begin{IEEEbiographynophoto}{Lifeng Lai} (SM'19) received the B.E. and M.E. degrees from Zhejiang University, Hangzhou, China in 2001 and 2004 respectively, and the Ph.D. from The Ohio State University at Columbus, OH, in 2007. He was a postdoctoral research associate at Princeton University from 2007 to 2009, an assistant professor at University of Arkansas, Little Rock from 2009 to 2012, and an assistant professor at Worcester Polytechnic Institute from 2012 to 2016. Since 2016, he has been an associate professor at University of California, Davis. Dr. Lai's research interests include information theory, stochastic signal processing and their applications in wireless communications, security and other related areas.
	
	Dr. Lai was a Distinguished University Fellow of the Ohio State University from 2004 to 2007. He is a co-recipient of the Best Paper Award from IEEE Global Communications Conference (Globecom) in 2008, the Best Paper Award from IEEE Conference on Communications (ICC) in 2011 and the Best Paper Award from IEEE Smart Grid Communications (SmartGridComm) in 2012. He received the National Science Foundation CAREER Award in 2011, and Northrop Young Researcher Award in 2012. He served as a Guest Editor for IEEE Journal on Selected Areas in Communications, Special Issue on Signal Processing Techniques for Wireless Physical Layer Security from 2012 to 2013, and served as an Editor for IEEE Transactions on Wireless Communications from 2013 to 2018. He is currently serving as an Associate Editor for IEEE Transactions on Information Forensics and Security.
\end{IEEEbiographynophoto}

\begin{IEEEbiographynophoto}{Erhan Bayraktar}, the holder of the Susan Smith Chair, is a full professor of Mathematics at the University of Michigan, where he has been since 2004 upon getting his Ph.D. from Princeton University. Professor Bayraktar's research is in stochastic analysis, control, applied probability, and mathematical finance. In particular, recently his research focused on mean field games, machine learning and model uncertainty. He has over 130 publications in top journals in these areas.
	
	Professor Bayraktar is recognized as a leader in his areas of research: He is a corresponding editor in the SIAM Journal on Control and Optimization and also serves in the editorial boards of Applied Mathematics and Optimization, Mathematics of Operations Research, Mathematical Finance. His research has been also been continually funded by the National Science Foundation. In particular, he received a CAREER grant. He has been a plenary speaker in numerous conferences and workshops. Professor Bayraktar has had 23 post-docs and 13 Ph.D. students, 9 of whom graduated. They hold prestigious positions in academia and industry.
	
	Professor Bayraktar has been the director of the Risk Management and Quantitative Finance Masters program at the University of Michigan since its inception in 2015.  
\end{IEEEbiographynophoto}

\end{document}